\documentclass[11pt]{article}

\usepackage[sc]{mathpazo}
\usepackage[margin=1in]{geometry}
\usepackage[utf8x]{inputenc}

\usepackage[usenames,dvipsnames]{xcolor}
\definecolor{Gred}{RGB}{219, 50, 54}
\definecolor{Ggreen}{RGB}{60, 186, 84}
\definecolor{Gblue}{RGB}{72, 133, 237}
\definecolor{Gyellow}{RGB}{247, 178, 16}
\definecolor{ToCgreen}{RGB}{0, 128, 0}
\definecolor{myGold}{RGB}{231,141,20}
\definecolor{myBlue}{rgb}{0.19,0.41,.65}
\definecolor{myPurple}{RGB}{175,0,124}
\definecolor{niceRed}{RGB}{153,0,0}

\usepackage{cmap}
\usepackage[T1]{fontenc}
\usepackage{bm}
\usepackage{amsmath}
\usepackage{amsfonts}
\usepackage{amssymb}
\usepackage{amsbsy}
\usepackage{amsthm}

\usepackage{graphicx, ucs}

\usepackage{subcaption}
\usepackage{rotating}
\usepackage{float}
\usepackage{tikz}

\usepackage{algorithm, caption}
\usepackage[noend]{algpseudocode}
\usepackage{listings}

\usepackage{enumitem}

\usepackage{hyperref}
\hypersetup{
  colorlinks = true,
  urlcolor = {blueGrotto},
  linkcolor = {royalBlue},
  citecolor = {limeGreen}
}

\usepackage{multirow}
\usepackage{array}

\usepackage{chngcntr}

\counterwithin*{equation}{section}
 \usepackage{chngcntr}
\usepackage{soul}
\usepackage{dsfont}
\usepackage{nicefrac}

\def\compactify{\itemsep=0pt \topsep=0pt \partopsep=0pt \parsep=0pt}
\let\latexusecounter=\usecounter

\newenvironment{Enumerate}
  {\def\usecounter{\compactify\latexusecounter}
   \begin{enumerate}}
  {\end{enumerate}\let\usecounter=\latexusecounter}

\def\eps{\varepsilon}

\DeclareMathOperator*{\argmin}{argmin}

\def\abs#1{\left|#1\right|}

\newcommand{\paragr}[1]{\noindent \textbf{#1}}

\def\matr#1{\boldsymbol{#1}}
\renewcommand{\vec}[1]{\boldsymbol{#1}}

\def\norm#1{\left\|#1\right\|}

\def\poly{\mathrm{poly}}

\def\Prob{\mathbb{P}}

\def\Exp{\mathbb{E}}
\def\Var{\mathrm{Var}}
\def\Cov{\mathrm{Cov}}
\def\normal{\mathcal{N}}

\newcommand{\reals}{\mathbb{R}}

\newtheorem{theorem}{Theorem}

\newtheorem{lemma}{Lemma}

\newtheorem{corollary}{Corollary}

\newtheorem{assumption}{Assumption}
\newtheorem{definition}{Definition}

\newenvironment{prevproof}[2]{\noindent {\bf {Proof of {#1}~\ref{#2}:}}}{$\blacksquare$\vskip \belowdisplayskip}

\definecolor{myC}{rgb}{0, 255, 255}
\definecolor{myY}{rgb}{204, 204, 0}
\definecolor{myM}{rgb}{255, 0, 255}
\definecolor{secinhead}{RGB}{249,196,95}
\definecolor{lgray}{gray}{0.8}

\newif\ifnotes\notestrue

\ifnotes
\usepackage{color}
\definecolor{mygrey}{gray}{0.50}
\newcommand{\notename}[2]{{\textcolor{red}{\footnotesize{\bf (#1:} {#2}{\bf
) }}}}

\else

\newcommand{\notename}[2]{{}}

\fi

 \definecolor{niceRed}{RGB}{190,38,38}
\definecolor{blueGrotto}{HTML}{059DC0}
\definecolor{royalBlue}{HTML}{057DCD}
\definecolor{navyBlueP}{HTML}{0B579C}
\definecolor{limeGreen}{HTML}{81B622}

\newcommand{\memb}{M}
\newcommand{\chara}{\mathds{1}}
\newcommand{\Hessian}{\mathbf{H}}

\newcommand{\symm}{\mathcal{Q}}

\newcommand{\Domain}{\mathcal{D}}

\begin{document}

\title{Computationally and Statistically Efficient Truncated  Regression}
\author{
  \textbf{Constantinos Daskalakis} \\
  \small Massachusetts Institute of Technology \\
  \url{costis@csail.mit.edu}
  \and
  \textbf{Themis Gouleakis} \\
  \small Max Plank Institute for Informatics \\
  \url{themis.gouleakis@gmail.com}
  \and
  \textbf{Christos Tzamos} \\
  \small University of Wisconsin-Madison \\
  \url{tzamos@wisc.edu}
  \and
  \textbf{Manolis Zampetakis} \\
  \small Massachusetts Institute of Technology \\
  \url{mzampet@mit.edu}
}
\maketitle
\thispagestyle{empty}

\begin{abstract}
    We provide a computationally and statistically efficient estimator for the
  classical problem of truncated linear regression, where the dependent
  variable $y = \vec{w}^{T} \vec{x} + \eps$ and its corresponding
  vector of covariates $\vec{x} \in \mathbb{R}^k$ are only revealed if the
  dependent variable falls in some subset $S \subseteq \mathbb{R}$; otherwise
  the existence of the pair $(\vec{x}, y)$ is hidden. This problem has remained
  a challenge since the early works of
  \cite{tobin1958estimation,amemiya1973regression,hausman1977social}, its
  applications are abundant, and its history dates back even further to the
  work of Galton, Pearson, Lee, and Fisher
  \cite{Galton1897,PearsonLee1908,fisher31}. While consistent
  estimators of the regression coefficients have been identified, the error
  rates are not well-understood, especially in high dimensions.

    Under a ``thickness assumption'' about the covariance matrix of the
  covariates in the revealed sample, we provide a computationally efficient
  estimator for the coefficient vector $\vec{w}$ from $n$ revealed samples that
  attains $\ell_2$ error $\tilde{O}(\sqrt{k / n})$. Our estimator uses Projected
  Stochastic Gradient Descent (PSGD) without replacement on the negative
  log-likelihood of the truncated sample. For the statistically efficient
  estimation we only need an oracle access to the set $S$, which may otherwise
  be arbitrary. In order to achieve computational efficiency also we need to
  assume that $S$ is a union of a finite number of intervals but still can be
  very complicated. PSGD without replacement must be restricted to an
  appropriately defined convex cone to guarantee that the negative
  log-likelihood is strongly convex, which in turn is established using
  concentration of matrices on variables with sub-exponential tails. We perform
  experiments on simulated data to illustrate the accuracy of our estimator.

    As a corollary of our work, we show that SGD provably learns the parameters
  of single-layer neural networks with noisy activation functions
  \cite{NairH10, BengioLC13, GulcehreMDB16}, given linearly many, in the number
  of network parameters, input-output pairs in the realizable setting.
\end{abstract}
\clearpage

\section{Introduction} \label{sec:intro}

  A central challenge in statistics is estimation from truncated samples.
Truncation occurs whenever samples that do not belong in some set $S$ are not
observed. For example, a clinical study of obesity will not contain samples
with weight smaller than a threshold set by the study. The related notion of
censoring is similar except that part of the sample may be observed even if
it does not belong to $S$. For example, the values that an insurance adjuster
observes are right-censored as clients report their loss as equal to the
policy limit when their actual loss exceeds the policy limit. In this case,
samples below the policy limit are shown, and only the count of the samples
that are above the limit is provided. Truncation and censoring have myriad
manifestations in business, economics, manufacturing, engineering, quality
control, medical and biological sciences, management sciences, social
sciences, and all areas of the physical sciences. As such they have received
extensive study.

  In this paper, we revisit the classical problem of
\textit{truncated linear regression}, which has been a challenge since the
early works of
\cite{tobin1958estimation, amemiya1973regression, hausman1977social, Maddala1986}.
Like standard linear regression, the dependent variable $y \in \mathbb{R}$ is
assumed to satisfy a linear relationship
$y = \vec{w}^{\rm T} \vec{x}+{\varepsilon}$ with the vector of covariates
$\vec{x} \in \mathbb{R}^k$, where $\varepsilon \sim {\cal N}(0,1)$, and
$\vec{w} \in \mathbb{R}^k$ is some unknown vector of regression coefficients.
Unlike standard linear regression, however, neither $\vec{x}$ nor $y$ are
observed, unless the latter belongs to some set $S \subseteq \mathbb{R}$.
Given a collection $(\vec{x}^{(i)},y^{(i)})_{i = 1, \ldots, n}$ of samples
that survived truncation, the goal is to estimate $\vec{w}$. In the closely
related and easier setting of \textit{censored linear regression}, we are
also given the set of covariates resulting in a truncated response.

  Applications of truncated and censored linear regression are abundant, as in
many cases observations are systematically filtered out during data
collection, if the response variable lies below or above certain thresholds.
An interesting example of truncated regression is discussed in
\cite{hausman1977social} where the effect of education and intelligence on
the earnings of workers in ``low level'' jobs is studied, based on a data
collected by surveying families whose incomes, during the year preceding the
experiment, were smaller than one and one-half times the 1967 poverty line.

  Truncated and censored linear regression have a long history, dating back to
at least \cite{tobin1958estimation}, and the work of
\cite{amemiya1973regression} and \cite{hausman1977social}.
\cite{amemiya1973regression} studies censored regression, when the
truncation set $S$ is a half-line and shows consistency and asymptotic
normality of the maximum likelihood estimator. He also proposes a two-step
Newton method to compute a consistent and asymptotically normal estimator.
Hausman and Wise study the harder problem of truncated regression also
establishing consistency of the maximum likelihood estimator. An overview of
existing work on the topic can be found in \cite{breen1996regression}. While
known estimators achieve asymptotic rate of $O_k({1/\sqrt{n}})$, at least for
sets $S$ that are half-lines, the dependence of $O_k(\cdot)$ on the dimension
$k$ is not well-understood. Moreover, while these weaker guarantees can be
attained for censored regression, no efficient algorithm is known at all for
truncated regression.

\smallskip Our goal in this work is to obtain computationally and
statistically efficient estimators for truncated linear regression. We make
no assumptions about the set $S$ that is used for truncation, except that we
are given oracle access to this set, namely, given a point $\vec{x}$ the
oracle outputs $\chara\{\vec{x} \in S\}$. We also make a couple of necessary
assumptions about the covariates of the samples.
\begin{description}
  \item[Assumption I:] the first is that the probability, conditionally on
        $\vec{x}^{(i)}$, that the response variable
        $y^{(i)} = \vec{w}^{\top} \vec{x}^{(i)} +\varepsilon^{(i)}$
        corresponding to a covariate $\vec{x}^{(i)}$ in our sample is not
        truncated is lower bounded by some absolute constant, say $1\%$, with
        respect to the choice of $\varepsilon^{(i)} \sim {\cal N}(0,1)$; this
        assumption is also necessary as was shown in \cite{DaskalakisGTZ18}
        for the special case of our problem, pertaining to truncated Gaussian
        estimation,
  \item[Assumption II:] the second is the same thickness assumption, also
        made in the some standard (untruncated) linear regression, that the
        average $\frac{1}{n} \sum_i \vec{x}^{(i)} \vec{x}^{(i) T}$ of the
        outer-products of the covariates in our sample has some absolute
        lower bound on its minimum singular value.
        \label{assumption costas 1}
\end{description}

  These assumptions are further discussed in Section \ref{sec:truncated}, and
in particular in Assumptions \ref{asp:survivalProbability} and
\ref{asp:logkAssumption}. As these are more cumbersome, the reader can just
think of Assumptions I and II stated above.

  Under Assumptions I and II (or Assumptions \ref{asp:survivalProbability}
and \ref{asp:logkAssumption}), and some boundness assumption on the parameter
vector, we provide the first time and sample efficient estimation algorithm for
truncated linear regression, whose estimation error of the coefficient vector
$\vec{w}$ decays as $\tilde{O}({\sqrt{k/n}})$. For a formal statement see
Theorem \ref{thm:estimationTheorem}. Our algorithm is the first computationally
efficient estimator for truncated linear regression. It is also the first, to
the best of our knowledge, estimator that can accommodate arbitrary truncation
sets $S$. This, in turn,  enables statistical estimation in settings where set
$S$ is determined by a complex set of rules, as it happens in many important
applications.

  We present below a high-level overview of the techniques involved in proving
our main result, and discuss further related work. Section \ref{sec:model}
provides the necessary preliminaries, while Section \ref{sec:truncated}
presents the truncated linear regression model and contains a discussion of
the assumptions made for our estimation. Section \ref{sec:estimation} states
our main result and provides its proof. In Section \ref{sec:neural} we
present an application of our estimation algorithm in learning the weights of
a single layer neural network with noisy activation function. Finally, in
Section \ref{sec:experiments}, we perform experiments on simulated data to
illustrate the accuracy of our method.
\medskip

\paragr{Learning Single-Layer Neural Networks with Noisy Activation
Functions.} Our main result implies as an immediate corollary the
learnability, via SGD, of single-layer neural networks with noisy Relu
activation functions \cite{NairH10, BengioLC13, GulcehreMDB16}. The noisy
Relu activations, considered in these papers for the purposes of improving
the stability of gradient descent, are similar to the standard Relu
activations, except that noise is added to their inputs before the
application of the non-linearity. In particular, if $z$ is the input to a
noisy Relu activation, its output is $\max\{0, z + \varepsilon\}$, where
$\varepsilon \sim \normal(0,1)$. In turn, a single-layer neural network with
noisy Relu activations is a random mapping,
$f_{\vec{w}}: \vec{x} \mapsto \max\{0, \vec{w}^T \vec{x} + \eps\}$, where
$\eps \sim \normal(0,1)$.

  We consider the learnability of single-layer neural networks of this type
in the realizable setting. In particular, given a neural network
$f_{\vec{w}}$ of the above form, and a sequence of inputs
$\vec{x}^{(1)}, \ldots, \vec{x}^{(n)}$, suppose that
$y^{(1)}, \ldots, y^{(n)}$ are the (random) outputs of the network on these
inputs. Given the collection $(\vec{x}^{(i)}, y^{(i)})_{i = 1}^n$ our goal is
to recover $\vec{w}$. This problem can be trivially reduced to the main
learning problem studied in this paper as a special case where: (i) the
truncation set is very simple, namely the half open interval $[0,+\infty)$;
and (ii) the identities of the inputs $\vec{x}^{(i)}$ resulting in truncation
are also revealed to us, namely we are in a censoring setting rather than a
truncation setting. As such, our more general results are directly applicable
to this setting. For more information see Section \ref{sec:neural}.

\subsection{Overview of the Techniques.} \label{sec:intro:techniques}

  We present a high-level overview of our time- and statistically-efficient
algorithm for truncated linear regression
(Theorem \ref{thm:estimationTheorem}). Our algorithm, shown in Figure
\ref{fig:algorithm}, is Projected Stochastic Gradient Descent (PSGD) without
replacement on the negative log-likelihood of the truncated samples. Notice
that we cannot write a closed-form expression for the negative log-likelihood,
as the set $S$ can be very complicated. Indeed, for the statistical efficiency
we only need to assume that we have oracle access to this set and can thus not
write down a formula for the measure of $S$ under different estimates of the
coefficient vector $\vec{w}$. While we cannot write a closed-form expression
for the negative log-likelihood, it is not hard to see that the negative
log-likelihood is still convex with respect to $\vec{w}$ for arbitrary
truncation sets $S \subseteq \mathbb{R}$.

  To effectively run the Stochastic Gradient Descent without replacement on the
negative log-likelihood, we need however to ensure that the algorithm remains
within a region where it is {\em strongly convex}. To accomplish this we define
a convex set of vectors $(\vec{w})$ in Definition \ref{def:regression:setaki}
and show in Theorem \ref{thm:conditionsSatisfied} that the negative
log-likelihood is strongly convex on that set; see in particular
\eqref{eq:thm:condition2} in the statement of the theorem, whose left hand side
is the Hessian of the negative log-likelihood. We also show that this set
contains the true coefficient vector in Lemma \ref{lem:thirdLemma}. Finally, we
show that we can efficiently project on this set; see Section
\ref{sec:projectionAlgorithmProof}.

  Thus we run our Projected Stochastic Gradient Descent without replacement
procedure on this set. As we have already noted, we have no closed-form
expression for the negative log-likelihood or its gradient. Nevertheless, we
show that, given oracle access to set $S$, we can get an un-biased sample of
the gradient. If $(\vec{x}_t, y_t)$ is a sample processed by PSGD without
replacement at step $t$, and $\vec{w}_t$ the current iterate, we perform
rejection sampling to obtain a sample from the Gaussian
$\normal(\vec{w}_t^T \vec{x}_t,1)$ conditioned on the truncation set $S$, in
order to compute an unbiased estimate of the gradient as per Eq.
\eqref{eq:gradientofLikelihoodFiniteSamples}. The rejection sampling that we use
could be computationally inefficient. For this reason to get a computationally
efficient algorithm we also assume that the set $S$ is a union of $r$
subintervals. In this case we can use a much faster sampling procedure and get
an efficient algorithm as a result, as we explain in Appendix
\ref{sec:unionIntervalsSampling}.

  Once we have established the strong convexity of the negative log-likelihood
as well as the efficient procedure to sample unbiased estimated of its gradient,
we are ready to analyze the performance of the PSGD without replacement. The
latter is a very challenging topic in the Machine Learning and only very
recently there have been works that analyze PSGD without replacement
\cite{Shamir16, PillaudRB18, NagarajJN19}. In this paper we use the framework
and the results of \cite{Shamir16} to analyze our algorithms.

\subsection{Further Related Work} \label{sec:further related}

  We have already surveyed work on truncated and censored linear regression
since the 1950s. Early precursors of this literature can be found in the
simpler, non-regression version of our problem, where the $\vec{x}^{(i)}$'s
are single-dimensional and equal, which corresponds to estimating a truncated
Normal distribution. This problem goes back to at least \cite{Galton1897},
\cite{Pearson1902}, \cite{PearsonLee1908}, and \cite{fisher31}. Following
these early works, there has been a large volume of research devoted to
estimating truncated Gaussians or other truncated distributions in one or
multiple dimensions; see e.g. \cite{Hotelling48,Tukey49}, and
\cite{Schneider86,Cohen91,BalakrishnanCramer} for an overview of this work.
There do exist consistent estimators for estimating the parameters of
truncated distributions, but, as in the case of truncated and censored
regression, the optimal estimation rates are mostly not well-understood. Only
very recent work of \cite{DaskalakisGTZ18} provides computationally and
statistically efficient estimators for the parameters of truncated
high-dimensional Gaussians. Similar to the present work,
\cite{DaskalakisGTZ18} we use PSGD to optimize the negative log-likelihood of
the truncated samples but with the main difference that in this paper we have
to use PSGD without replacement. Moreover, identifying the set where the
negative log-likelihood is strongly convex and establishing its strong
convexity are also simpler tasks in the truncated Gaussian setting compared to
the truncated regression setting, due to the shifting of the mean of the
samples induced by the different covariates $\vec{x}^{(i)}$.

  Last but not least our problem falls in the realm of robust Statistics, where
there has been a strand of recent works studying robust estimation and learning
in high dimensions. A celebrated result by \cite{CLMW11} computes the PCA of a
matrix, allowing for a constant fraction of its entries to be adversarially
corrupted, but they require the locations of the corruptions to be relatively
evenly distributed. Related work of \cite{XCM10} provides a robust PCA
algorithm for arbitrary corruption locations, requiring at most $50\%$ of the
points to be corrupted.

  \cite{DKK+16b,LRV16,DKK+17,DKK+18} do robust estimation of the parameters of
multi-variate Gaussian distributions in the presence of arbitrary corruptions to
a small $\varepsilon$ fraction of the samples, allowing for both deletions of
samples and additions of samples that can also be chosen adaptively (i.e.~after
seeing the sample generated by the Gaussian). The authors in \cite{CSV17} show
that corruptions of an arbitrarily large fraction of samples can be tolerated as
well, as long as we allow ``list decoding'' of the parameters of the Gaussian.
In particular, they design learning algorithms that work when an
$(1-\alpha)$-fraction of the samples can be adversarially corrupted, but output
a set of ${\rm poly}(1/\alpha)$ answers, one of which is guaranteed to be
accurate.

  Closer to our work in this strand of literature are works studying robust
linear regression \cite{BJK15,diakonikolas2018efficient} where a small fraction
of the response variables are arbitrarily corrupted. As we already discussed in
Section \ref{sec:further related}, these results allow arbitrary corruptions yet
a small number of them. We only allow filtering out observations, but an
arbitrarily large fraction of them.

  Other works in this literature include robust estimation under sparsity
assumptions \cite{Li17,BDLS17}. In \cite{HM13}, the authors study robust
subspace recovery having both upper and lower bounds that give a trade-off
between efficiency and robustness. Some general criteria for robust estimation
are formulated in \cite{SCV18}.

  For the most part, these works assume that an adversary perturbs a small
fraction of the samples \textit{arbitrarily}. Compared to truncation and
censoring, these perturbations are harder to handle. As such only small amounts
of perturbation can be accommodated, and the parameters cannot be estimated to
arbitrary precision. In contrast, in our setting the truncation set $S$ may
very well  have an ex ante probability of obliterating most of the
observations, say $99\%$ of them, yet the parameters of the model can still
be estimated to arbitrary precision.

\section{Preliminaries} \label{sec:model}

\paragr{Notation.} Let $\langle \vec{x}, \vec{y} \rangle$ be the inner product
of $\vec{x}, \vec{y} \in \reals^k$. We use $\matr{I}_k$ to refer to the identity
matrix in $k$ dimensions. We may drop the subscript when the dimensions are
clear. Let also $\symm_k$ be the set of all the symmetric $k \times k$
matrices. The covariance matrix between two vector random variables
$\vec{x}, \vec{y}$ is $\Cov[\vec{x}, \vec{y}]$.
\smallskip

\paragr{Vector and Matrix Norms.} We define the $\ell_p$-norm of
$\vec{x} \in \reals^{k}$ to be
$\norm{\vec{x}}_p = \left( \sum_i x_i^p \right)^{1/p}$ and the
$\ell_{\infty}$-norm of $\vec{x}$ to be
$\norm{\vec{x}}_{\infty} = \max_{i} \abs{x_i}$. We also define the
spectral norm of a matrix $\matr{A}$ to be
\[ \norm{\matr{A}}_2 = \max_{\vec{x} \in \reals^d, \vec{x}
\neq 0} \frac{\norm{\matr{A} \vec{x}}_2}{\norm{\vec{x}}_2}. \]
It is well known that for $\matr{A} \in \symm_k$,
$\norm{\matr{A}}_2 = \max_{i} \{\abs{\lambda_i}\}$, where $\lambda_i$'s are the
eigenvalues of $\matr{A}$. The Frobenius norm of a matrix
$\matr{A} = (a_{ij}) \in \symm_d$ is defined as $\norm{\matr{A}}_F=\sqrt{\sum_{i,j} a_{ij}^2}$.
\smallskip

\paragr{Truncated Gaussian Distribution.} Let
$\normal(\vec{\mu}, \matr{\Sigma})$ be the normal distribution with mean
$\vec{\mu}$ and covariance matrix $\matr{\Sigma}$, with the following
probability density function
\begin{align} \label{eq:normalDensityFunction}
  \normal(\vec{\mu}, \matr{\Sigma}; \vec{x}) =
  \frac{1}{\sqrt{\det(2 \pi \matr{\Sigma})}}
  \exp \left( - \frac{1}{2} (\vec{x} - \vec{\mu})^T \matr{\Sigma}^{-1}
  (\vec{x} - \vec{\mu}) \right).
\end{align}
Also, let $\normal(\vec{\mu},\matr{\Sigma}; S)$ denote the \emph{probability mass of a set $S$} under this Gaussian measure. Let $S \subseteq \reals^k$ be a
subset of the $k$-dimensional Euclidean space, we define the
\textit{$S$-truncated normal distribution}
$\normal(\vec{\mu}, \matr{\Sigma}, S)$ the normal distribution
$\normal(\vec{\mu}, \matr{\Sigma})$ conditioned on taking values in the subset
$S$. The probability density function of $\normal(\vec{\mu}, \matr{\Sigma}, S)$
is the following
\begin{equation} \label{eq:truncatedNormalDensityFunction}
  \normal(\vec{\mu}, \matr{\Sigma}, S; \vec{x}) =
    \begin{cases}
      \frac{1}{\normal(\vec{\mu},\matr{\Sigma}; S)}
      \cdot \normal(\vec{\mu}, \matr{\Sigma}; \vec{x}) & ~~~ \vec{x} \in S \\
      0                                                & ~~~ \vec{x} \not\in S
    \end{cases} .
\end{equation}

\paragr{Membership Oracle of a Set.} Let $S \subseteq \reals^k$ be a subset
of the $k$-dimensional Euclidean space. A \textit{membership oracle} of $S$
is an efficient procedure $\memb_S$ that computes the characteristic function
of $S$, i.e. $\memb_S(\vec{x}) = \chara\{\vec{x} \in S\}$.

\section{Truncated Linear Regression Model} \label{sec:truncated}

  Let $S \subseteq \reals$ be a measurable subset of the real line.
We assume that we have access to $n$ truncated samples of the form
$(\vec{x}^{(i)}, y^{(i)})$. Truncated samples are generated as follows:
\begin{Enumerate}
  \item one $\vec{x}^{(i)} \in \reals^k$ is picked arbitrarily,
  \item the value $y^{(i)}$ is computed according to
        \begin{equation} \label{eq:truncatedLinearRegressionDefinition}
          y^{(i)} = \vec{w}^{* T} \vec{x}^{(i)} + \eps^{(i)},
        \end{equation}
        where $\eps^{(i)}$ is sampled from a standard normal distribution
        $\normal(0, 1)$,
  \item if $y^{(i)} \in S$ then return $(\vec{x}^{(i)}, y^{(i)})$, otherwise
        repeat from step 1 with the same index $i$.
\end{Enumerate}
\bigskip

  Without any assumptions on the truncation set $S$, it is easy to see that no
meaningful estimation is possible. When $k = 1$ the regression problem becomes
the estimation of the mean of a Gaussian distribution that has been studied in
\cite{DaskalakisGTZ18}. In this case the necessary and sufficient condition is
that the Gaussian measure of the set $S$ is at least a constant $\alpha$. When
$k > 1$ though, for every $\vec{x} \in \reals^k$ we have a different
$\alpha(\vec{x})$ defined as we can see in the following definition.

\begin{definition}[\textsc{Survival Probability}]
\label{def:survivalProbability}
    Let $S$ be a measurable subset of $\reals$. Given $\vec{x} \in \reals^k$ and
  $\vec{w} \in \reals^{k}$ we define the \textit{survival probability}
  $\alpha(\vec{w}, \vec{x}; S)$ of the sample with feature vector $\vec{x}$ and
  parameters $\vec{w}$ as
  \[ \alpha(\vec{w}, \vec{x}; S) = \normal(\vec{w}^T \vec{x}, 1; S). \]
  When $S$ is clear from the context we may refer to
  $\alpha(\vec{w}, \vec{x}; S)$ simply as $\alpha(\vec{w}, \vec{x})$.
\end{definition}

\noindent Since $S$ has a different mass for every $\vec{x}$, the assumption
that we need in this regime is more complicated than the assumption used by
\cite{DaskalakisGTZ18}. A natural candidate assumption is that for every
$\vec{x}^{(i)}$ the mass $\alpha(\vec{w}, \vec{x}^{(i)})$ is large enough. We
propose an even weaker condition which is sufficient for recovering the
regression parameters and only lower bounds an average of
$\alpha(\vec{w}, \vec{x}^{(i)})$.

\begin{assumption}[\textsc{Constant Survival Probability Assumption}]
\label{asp:survivalProbability}
    Let $(\vec{x}^{(1)}, {y}^{(1)})$, $\dots$, $(\vec{x}^{(n)}, {y}^{(n)})$
  be samples from the regression model
  \eqref{eq:truncatedLinearRegressionDefinition}. There exists a constant
  $a > 0$ such that
  \[ \sum_{i = 1}^n \log \left(\frac{1}{\alpha(\vec{x}^{(i)}, \vec{w}^*)}\right)
     \vec{x}^{(i)} \vec{x}^{(i) T} \preceq \log\left(\frac{1}{a}\right)
     \sum_{i = 1}^n \vec{x}^{(i)} \vec{x}^{(i) T}. \]
\end{assumption}

\noindent Our second assumption involves only the $\vec{x}^{(i)}$'s that we
observe and is similar to the usual assumption in linear regression that
covariance matrix of $\vec{x}^{(i)}$'s has high enough variance in every
direction.

\begin{assumption}[\textsc{Thickness of Covariance Matrix of Covariates Assumption}]\footnote{We want
to highlight that the assumption $\matr{X} \succeq b^2 \cdot \matr{I}$ in
Assumption \ref{asp:logkAssumption} can be eliminated if we allow the error to
be measured in the Mahalanobis distance with scale matrix $\matr{X}$. We
prefer to keep this assumption throughout the paper due to the simplicity of
exposition of the formal statements and the proofs but it is not hard to see
that all our proofs generalize.} \label{asp:logkAssumption}
    Let $\matr{X}$ be the $k \times k$ matrix defines as
  $\matr{X} = \frac{1}{n} \sum_{i = 1}^n \vec{x}^{(i)} \vec{x}^{(i) T}$,
  where $\vec{x}^{(i)} \in \reals^k$. Then for every $i \in [n]$, it holds that
  \[ \matr{X} \succeq \frac{\log(k)}{n} \vec{x}^{(i)} \vec{x}^{(i) T}, \quad \text{and} \quad \matr{X} \succeq b^2 \cdot \matr{I} \]
  for some value $b \in \reals_+$.
\end{assumption}

\noindent The aforementioned thickness assumption can also be replaced by the
assumption (1) $\matr{X} \succeq \matr{I}$, and (2)
$\norm{\vec{x}^{(i)}}_2^2 \le \frac{n}{\log(k)}$, for all $i \in [n]$. This pair
of assumptions hold with high probability if the covariates are sampled from
some well-behaved distribution, e.g. a multi-dimensional Gaussian.

\section{Estimating the Parameters of a Truncated Linear Regression}
\label{sec:estimation}

  We start with the formal statement of our main theorem about the parameter
estimation of the truncated linear regression model
\eqref{eq:truncatedLinearRegressionDefinition}.

\begin{theorem} \label{thm:estimationTheorem}
    Let $(\vec{x}^{(1)}, y^{(1)}), \dots, (\vec{x}^{(n)}, y^{(n)})$ be $n$
  samples from the linear regression model
  \eqref{eq:truncatedLinearRegressionDefinition} with parameters $\vec{w}^*$,
  such that $\norm{\vec{x}^{(i)}}_{\infty} \le B$, and
  $\norm{\vec{w}^*}_{2} \le C$. If Assumptions
  \ref{asp:survivalProbability} and \ref{asp:logkAssumption} hold, then there
  exists an algorithm with success probability at least $2/3$, that outputs
  $\hat{\vec{w}} \in \reals^k$ such that
  \[ \norm{\hat{\vec{w}} - \vec{w}^*}_2 \le \poly(1/a)
  \cdot \frac{B \cdot C}{b^2} \cdot \sqrt{\frac{k}{n} \log(n)}. \]
  \noindent Moreover if the truncation set $S$ is a union of $r$ intervals then
  the algorithm runs in time $\poly(n, r, k, 1/a)$.
\end{theorem}

  As we explained in the introduction, our estimation algorithm is Projected
Stochastic Gradient Descent without replacement for maximizing the population
log-likelihood function with the careful choice of the projection set. We first
present an outline of the proof of Theorem \ref{thm:estimationTheorem} and then
we present the individual lemmas that complete every step of the proof.

  The framework that we use for to analyze the Projected Stochastic Gradient
Descent without replacement is based on the paper by Ohad Shamir
\cite{Shamir16}. We start by presenting this framework adapted so that is fits
with our our truncated regression setting.

\subsection{Stochastic Gradient Descent Without Replacement}
\label{sec:sgdWithoutReplacement}

  Let  $f : \reals^k \to \reals$ be a convex function of the form
\[ f(\vec{w}) = \frac{1}{n} \sum_{i = 1}^n f_i(\vec{w}), \]
\noindent where we assume that the function $f_i$ are enumerated in a random
order. The following algorithm describes projected stochastic gradient descent
without replacement applied to $f$, with projection set
$\Domain \subseteq \reals^k$. We define formally the particular set $\Domain$
that we consider in Definition \ref{def:regression:setaki}. For now $\Domain$
should be thought of an arbitrary convex subset of $\reals^k$.

\begin{algorithm}[H]
\caption*{\textbf{Algorithm ($\star$). Projected SGD Without Replacement for $\boldsymbol{\lambda}$-Strongly Convex Functions.}}
\begin{algorithmic}[1]
  \State $\vec{w}^{(0)} \gets$ arbitrary point in $\Domain$ \Comment{(a) \textit{initial feasible point}}
  \For{$i = 1, \dots, n$}
    \State Sample $\vec{v}^{(i)}$ such that $\Exp\left[ \vec{v}^{(i)} \mid \vec{w}^{(i - 1)}\right] = \nabla f_i(\vec{w}^{(i - 1)})$ \Comment{(b) \textit{estimation of gradient}}
    \State $\vec{r}^{(i)} \gets \vec{w}^{(i - 1)} - \frac{1}{\lambda \cdot i} \vec{v}^{(i)}$
    \State $\vec{w}^{(i)} \gets \argmin_{\vec{w} \in \Domain} \norm{\vec{w} - \vec{r}^{(i)}}$ \Comment{(c) \textit{projection step}}
  \EndFor \label{euclidendwhile}
  \State \textbf{return} $\bar{\vec{w}} \gets \frac{1}{M} \sum_{i = 1}^M \vec{w}^{(i)}$
\end{algorithmic}
\label{alg:projectedSGDGeneral}
\end{algorithm}

  Our goal is to apply the Algorithm
\hyperref[alg:projectedSGDGeneral]{($\star$)} to the negative log-likelihood
function of the truncated linear regression model. It is clear from the above
description that in order to apply Algorithm
\hyperref[alg:projectedSGDGeneral]{($\star$)} we have to solve the following
three algorithmic problems
\begin{enumerate}
  \item[(a)] \textbf{initial feasible point:} efficiently compute an initial
             feasible point in $\Domain$,
  \item[(b)] \textbf{unbiased gradient estimation:} efficiently sample an
             unbiased estimation of each $\nabla f_i$,
  \item[(c)] \textbf{efficient projection:} design an efficient algorithm to
             project to the set $\Domain$.
\end{enumerate}
Solving (a) - (c) is the first step in the proof of Theorem
\ref{thm:estimationTheorem}. Then our goal is to apply the following Theorem 3
of \cite{Shamir16}.

\begin{theorem}[Theorem 3 of \cite{Shamir16}] \label{thm:mainSGDAnalysis}
	  Let $f : \reals^k \to \reals$ be a convex function, such that
  $f(\vec{w}) = \frac{1}{n} \sum_{i = 1}^n f_i(\vec{w})$ where
  $f_i(\vec{w}) = c_i \cdot \vec{w}^T \vec{x}^{(i)} + q(\vec{w})$,
  $c_i \in \reals$ and $\vec{x}^{(i)} \in \reals^k$ with
  $\norm{\vec{x}^{(i)}}_2 \le 1$. Let also $\vec{w}^{(1)}, \dots, \vec{w}^{(n)}$ be
  the sequence produced by Algorithm
  \hyperref[alg:projectedSGDGeneral]{($\star$)} where
  $\vec{v}^{(1)}, \dots, \vec{v}^{(n)}$ is a sequence of random vectors such
  that
  $\Exp\left[ \vec{v}^{(i)} \mid \vec{w}^{(i - 1)}\right] = \nabla f_i(\vec{w}^{(i - 1)})$
  for all $i \in [n]$ and
  $\vec{w}^* = \arg \min_{\vec{w} \in \Domain} f(\vec{w})$ be a minimizer
  of $f$. If we assume the following:
  \begin{enumerate}
    \item[\emph{(i)}] \emph{\textbf{bounded variance step:}}
                      $\Exp\left[ \norm{\vec{v}^{(i)}}_2^2 \right] \le \rho^2$,
    \item[\emph{(ii)}] \emph{\textbf{strong convexity:}} $q(\vec{w})$ is
                      $\lambda$-strongly convex,
    \item[\emph{(iii)}] \emph{\textbf{bounded parameters:}} the diameter of
                      $\Domain$ is at most $\rho$ and also
                      $\max_i \abs{c_i} \le \rho$,
  \end{enumerate}
  \noindent then,
  $\Exp\left[ f(\vec{\bar{w}}) \right] - f(\vec{w}^*) \le c \cdot \frac{\rho^2}{\lambda n} \cdot \log(n),$
  where $\bar{\vec{w}}$ is the output of the Algorithm
  \hyperref[alg:projectedSGDGeneral]{($\star$)} and $c \in \reals_+$.
\end{theorem}

  Theorem \ref{thm:mainSGDAnalysis} is essentially the same as Theorem 3 of
\cite{Shamir16}, slightly adapted to fit to our problem. The bigger difference
is that in \cite{Shamir16} the variable $\vec{v}^{(i)}$ is exactly equal to the
gradient $\nabla f_i(\vec{w}^{(i - 1)})$ instead of being an unbiased estimate
of $\nabla f_i(\vec{w}^{(i - 1)})$. It is easy to check in Section 6.5 of
\cite{Shamir16} that this slight difference does not change the proof and the
above theorem holds.

  As we can see from the expression of Theorem \ref{thm:mainSGDAnalysis} one
bottleneck is that it applies only in the setting where
$\norm{\vec{x}^{(i)}}_2 \le 1$. To solve our problem in our more general setting
where we have only assumed that $\norm{\vec{x}^{(i)}}_{\infty} \le B$, we make
sure, before running the algorithm, to divide all the covariates $\vec{x}^{(i)}$
by $B \sqrt{k}$. This way the norm of $\vec{w}$ will correspondingly multiplied
by $B \sqrt{k}$. So for the rest of the proof we may replace the pair of
assumptions $\norm{\vec{x}^{(i)}}_{\infty} \le B$ and $\norm{\vec{w}}_2 \le C$
of Theorem \ref{thm:estimationTheorem} with the pair of assumptions
\begin{align} \label{eq:normalizedAssumptions}
  \norm{\vec{x}^{(i)}}_2 \le 1 \quad \mathrm{and} \quad \norm{\vec{w}}_2 \le \bar{B}
\end{align}
\noindent where $\bar{B} = B \cdot C \cdot \sqrt{k}$.

\subsection{Outline of Proof of Theorem \ref{thm:estimationTheorem}}
\label{sec:outline}

  In this section we outline how to use Theorem \ref{thm:mainSGDAnalysis} for
the estimation of the parameters of the truncated linear regression that we
described in the Section \ref{sec:truncated}. Our abstract goal is to maximize
the \textit{population} log-likelihood function using projected stochastic
gradient descent without replacement. We start with the definition of the
log-likelihood function and we proceed in the next sections with the necessary
lemmas to prove the algorithmic properties (a) - (c) and the statistical
properties (i) - (iv)\footnote{Property (iv) is not discussed yet, but we
explain it later in this section.} that allow us to use Theorem
\ref{thm:mainSGDAnalysis}.
\medskip

  We first present the negative log-likelihood of a single sample and then we
present the population version of the negative log-likelihood function and its
first two derivatives.

  Given the sample $(\vec{x}, y) \in \reals^k \times \reals$, the
log-likelihood that $(\vec{x}, y)$ is a sample of the form the truncated linear
regression model \eqref{eq:truncatedLinearRegressionDefinition}, with survival
set $S$ and parameters $\vec{w}$ is equal to
\begin{align} \label{eq:logLikelihoodOneSample}
  \ell(\vec{w}; \vec{x}, y) \triangleq & - \frac{1}{2} (y - \vec{w}^T \vec{x})^2 - \log \left( \int_S \exp \left( - \frac{1}{2} (z - \vec{w}^T \vec{x})^2 \right) dz \right) \nonumber \\
                            = & - \frac{1}{2} y^2 + y \cdot \vec{w}^T \vec{x}
                                - \log \left( \int_S \exp \left(-
                                      \frac{1}{2} z^2 + z \cdot \vec{w}^T
                                      \vec{x}
                                  \right) dz \right)
\end{align}

\noindent The population log-likelihood function with $n$ samples is equal to

\begin{align} \label{eq:logLikelihoodFiniteSamples}
\bar{\ell}(\vec{w}) \triangleq \frac{1}{n} \sum_{i = 1}^n \Exp_{y \sim \normal(\vec{w}^{* T} \vec{x}^{(i)}, 1, S)} \left[ \ell(\vec{w}; \vec{x}^{(i)}, y) \right].
\end{align}

\noindent We now compute the gradient of $\bar{\ell}(\vec{w})$.

\begin{align} \label{eq:gradientofLikelihoodFiniteSamples}
\nabla \bar{\ell}(\vec{w}) = & ~
  \frac{1}{n}
    \sum_{i = 1}^n
      \Exp_{y \sim \normal(\vec{w}^{* T} \vec{x}^{(i)}, 1, S)} \left[
        y \cdot \vec{x}^{(i)}     \right]
  -   \frac{1}{n}
  \sum_{i = 1}^n
    \Exp_{z \sim \normal(\vec{w}^T \vec{x}^{(i)}, 1, S)}
      \left[z \cdot \vec{x}^{(i)}    \right].
\end{align}

\noindent Finally, we compute the Hessian $\Hessian_{\bar{\ell}}$
\begin{align} \label{eq:HessianofLikelihoodFiniteSamples}
\Hessian_{\bar{\ell}}
        = & - \frac{1}{n}
            \sum_{i = 1}^n
              \Cov_{z \sim \normal(\vec{w}^T \vec{x}^{(i)},
                   1, S)} \left[z \cdot \vec{x}^{(i)} , z \cdot \vec{x}^{(i)} \right].
\end{align}

  Since the covariance matrix of a random variable is always positive
semidefinite, we conclude that $\Hessian_{\bar{\ell}}$ is negative
semidefinite, which implies the following lemma.

\begin{lemma} \label{lem:populationLogLikelihoodIsConcave}
    The population log-likelihood function $\bar{\ell}(\vec{w})$ is a concave
  function.
\end{lemma}

  Our goal is to use Theorem \ref{thm:mainSGDAnalysis} with
$f(\vec{w})$ equal to $- \bar{\ell}(\vec{w})$ but ignoring the parts that do
not have any dependence of $\vec{w}$. For our analysis to work we choose we the
following decomposition of $f$ into a sum of the following $f_i$'s
\begin{align} \label{eq:logLikelihoodDecomposition}
  f_i(\vec{w}) & = \Exp_{y \sim \normal(\vec{w}^{*T} \vec{x}^{(i)}, 1, S)}\left[ - y \cdot \vec{w}^T \vec{x}^{(i)} \right] + \frac{1}{n} \sum_{i = 1}^n \log\left(\int_S \exp\left( - \frac{1}{2} z^2 + z \cdot \vec{w}^T \vec{x}^{(i)} \right) dz\right).
\end{align}
It is easy to see that the above functions $f_i$ are in the form of Theorem
\ref{thm:mainSGDAnalysis} with
\[ c_i \triangleq - \Exp_{y \sim \normal(\vec{w}^{*T} \vec{x}^{(i)}, 1, S)}\left[y \right] \quad \text{ and } \]
\[ q(\vec{w}) \triangleq \frac{1}{n} \sum_{i = 1}^n \log\left(\int_S \exp\left( - \frac{1}{2} z^2 + z \cdot \vec{w}^T \vec{x}^{(i)} \right) dz\right). \]

  Unfortunately, none of the properties (i) - (iii) of $f$ hold for all vectors
$\vec{w} \in \reals^k$ and for this reason, we add the projection step. We
identify a projection set $\Domain_{r^*, \bar{B}}$ such that log-likelihood
satisfies both (i) and (ii) for all vectors
$\vec{w} \in \Domain_{r^*, \bar{B}}$.

\begin{definition}[\textsc{Projection Set}] \label{def:regression:setaki}
  We define
  \[ \Domain_{r, B} = \left\{ \vec{w} \in \reals^{k} \mid
     \sum_{i = 1}^n \left(y^{(i)} - \vec{w}^T \vec{x}^{(i)}\right)^2
     \vec{x}^{(i)} \vec{x}^{(i) T} \preceq r \sum_{i = 1}^n \vec{x}^{(i)} \vec{x}^{(i) T} \text{ and } \norm{\vec{w}}_2 \le B \right\}. \]
  We set $r^* = 4 \log(2/a) + 7$ and $\bar{B}$ from
  \eqref{eq:normalizedAssumptions}. We say that $\vec{x} \in \reals^k$ is
  \textbf{feasible} if and only if $\vec{x} \in \Domain_{r^*, \bar{B}}$.
\end{definition}

\noindent Using the projection set $\Domain_{r^*, \bar{B}}$, we can prove (i),
(ii) and (iii) and hence we can apply Theorem \ref{thm:mainSGDAnalysis}. The
last step is to transform the conclusions of Theorem \ref{thm:mainSGDAnalysis}
to guarantees in the parameter space. For this we use again the strong convexity
of $f$ which implies that closeness in the objective value translates to
closeness in the parameter space. For the latter we also need the following
property:
\begin{enumerate}
  \item[(iv)] \textbf{feasibility of optimal solution:}
              $\vec{w}^* \in \Domain_{r^*, \bar{B}}$.
\end{enumerate}

With these definitions in mind we are ready to sketch how to solve the
algorithmic problems (a) - (c) and prove the statistical properties (i) -
(iv). For the problem (a) we observe that is reducible to (c) since once we
have an efficient procedure to project we can start from an arbitrary point in
$\reals^k$, e.g. $\vec{w} = \vec{0}$ and project to $\Domain_{r^*, B}$ and this
is our initial point.
\begin{enumerate}
          \item In Section \ref{sec:technical} we provide some technical lemmas that
    are useful to understand the formal statements of the rest of the proof.
  \item In Section \ref{sec:gradient} we present the details of the
    Algorithm \ref{alg:rejectionSampling} that is used to compute an unbiased
    estimation of the gradient, which gives a solution to the algorithmic
    problem (b).
  \item In Section \ref{sec:projectionAlgorithmProof} we present a
    detailed analysis of our projection Algorithm
    \ref{alg:projectionAlgorithm}, which gives a solution to algorithmic
    problems (a) and (c).
  \item In Section \ref{sec:statistical} we present the statements that
    prove the (i) bounded variance and (ii) strong convexity of the
    log-likelihood function. This is the main technical contribution of the
    paper and uses all the results that we have proved in Section
    \ref{sec:technical} together with Assumptions
    \ref{asp:survivalProbability}, \ref{asp:logkAssumption}.
  \item In Section \ref{sec:optimal} we prove the feasibility of the
    optimal solution, i.e. $\vec{w}^* \in \Domain_{r^*, B}$ which proves the
    property (iv).
  \item In Section \ref{sec:boundedParameters} we prove that the diameter
    of $\Domain_{r, B}$ is bounded and that the coefficient $c_i$ of the
    Theorem \ref{thm:mainSGDAnalysis} are also bounded, which proves the
    property (iii).
  \item Finally in Section \ref{sec:mainResultProof} we use all the mentioned
    results to prove our main Theorem \ref{thm:estimationTheorem}.
\end{enumerate}

\subsection{Survival Probability of Feasible Points} \label{sec:technical}

  One necessary technical lemma is how to correlate the survival probabilities
$\alpha(\vec{w}, \vec{x})$, $\alpha(\vec{w}', \vec{x})$ for two different
points $\vec{w}$, $\vec{w}'$. In Lemma \ref{lem:firstLemma} we show that this
is possible based on their distance with respect to $\vec{x}$. Then in Lemma
\ref{lem:secondLemma} we show how the expected second moment with respect to
the truncated Guassian error is related with the value of the corresponding
survival probability. We present the proofs of Lemma \ref{lem:firstLemma} and
Lemma \ref{lem:secondLemma} in the Appendix \ref{sec:firstLemmaProof} and
\ref{sec:secondLemmaProof} respectively.

\begin{lemma} \label{lem:firstLemma}
    Let $\vec{x}$, $\vec{x}'$, $\vec{w}$, $\vec{w}' \in \reals^{k}$, then
  $\log \left(\frac{1}{\alpha(\vec{w}, \vec{x})}\right) \leq 2 \log \left(\frac{1}{\alpha(\vec{w}', \vec{x})}\right) + \abs{\left( \vec{w} - \vec{w}' \right)^T \vec{x}}^2 + 2$
  and also
  $\log \left(\frac{1}{\alpha(\vec{w}, \vec{x})}\right) \leq 2 \log \left(\frac{1}{\alpha(\vec{w}, \vec{x}')}\right) + \abs{\vec{w}^T \left(\vec{x} - \vec{x}'\right)}^2 + 2$.
\end{lemma}

\begin{lemma} \label{lem:secondLemma}
    Let $\vec{x} \in \reals^k$, $\vec{w} \in \reals^{k}$, then
  $\Exp_{y \sim \normal(\vec{w}^T \vec{x}, 1, S)}\left[ (y-\vec{w}^T\vec{x})^2 \right] \le 2 \log \left( \frac{1}{\alpha(\vec{w}, \vec{x})} \right) + 4$.
\end{lemma}

  One corollary of these lemmas is an interesting property of feasible points,
i.e. points $\vec{w}$ inside $\Domain_{r^*, \bar{B}}$, namely that they satisfy
Assumption \ref{asp:survivalProbability}, under the assumption that
$\vec{w}^* \in \Domain_{r^*, \bar{B}}$, which we will prove later.

\subsection{Unbiased Gradient Estimation} \label{sec:gradient}

  Using \eqref{eq:gradientofLikelihoodFiniteSamples} we have that the gradient
of the function $f_i$ is equal to
\begin{equation} \label{eq:f_iGradient}
  \nabla f_i(\vec{w}) = \Exp_{y \sim \normal(\vec{w}^{* T} \vec{x}^{(i)}, 1, S)} \left[y \cdot \vec{x}^{(i)}\right] - \frac{1}{n} \sum_{j = 1}^n \Exp_{z \sim \normal(\vec{w}^T \vec{x}^{(j)}, 1, S)} \left[z \cdot \vec{x}^{(j)}\right].
\end{equation}
Hence an unbiased estimation of $\nabla f_i(\vec{w})$ can be computed given
one sample from the distribution $\normal(\vec{w}^{* T} \vec{x}^{(i)}, 1, S)$
and one sample from $\normal(\vec{w}^T \vec{x}^{(j)}, 1, S)$ where $j$ is
chosen uniformly at random from the set $[n]$. For the first sample we can
just use $y^{(i)}$. For the second sample though we need a sampling procedure
that given $\vec{w}$ and $\vec{x}$ produces a sample from
$\normal(\vec{w}^T \vec{x}, 1, S)$. For this we could simply use rejection
sampling, but because we have not assumed that $\alpha(\vec{w}, \vec{x})$ is
always large, we use a more elaborate argument starting with following Lemma
\ref{lem:rejectionSamplingLemma} whose proof is presented in the Appendix
\ref{sec:rejectionSamplingLemmaProof}.

\begin{lemma} \label{lem:rejectionSamplingLemma}
    If $\norm{\vec{x}^{(i)}}_2 \le 1$ for all $i \in [n]$ then for all
  $\vec{w} \in \Domain_{r^*, \bar{B}}$ it holds that
  \[ \alpha\left(\vec{w}, \vec{x}^{(i)}\right) \ge \poly(a) \cdot \exp\left(- O(\bar{B})\right). \]
\end{lemma}

\noindent Now if we do not care about computational efficiency we can just
apply rejection sampling using the membership oracle $\memb_S$. But if we make
the additional assumption of Theorem \ref{thm:estimationTheorem} that $S$ is a
union of $r$ intervals then once we have established that the survival
probability $\alpha(\vec{w}, \vec{x}^{(i)})$ is lower bounded by an
exponential on $\poly(B, k)$ we can use a much more efficient sampling
procedure that is based on accurate computations of the error function of a
Gaussian distribution. The latter is a simple algorithm that involves an
inverse transform sampling and we discuss the technical details in the
Appendix \ref{sec:unionIntervalsSampling}.

\subsection{Projection to the Feasible Set}
\label{sec:projectionAlgorithmProof}

  The convex problem we need to solve in this step is the following
  \begin{align} \label{eq:projectionProblem}
    \min_{\vec{z} \in \Domain_{r, B}} \norm{\vec{z} - \vec{w}}_2.
  \end{align}
  \noindent For simplicity in this section we may assume without loss of
  generality that $\sum_{i = 1}^n \vec{x}^{(i)} \vec{x}^{(i) T} = \matr{I}$.

  The main idea of the algorithm to solve \ref{eq:projectionProblem} is to use
the ellipsoid method with separating hyperplane oracle as descripted in Chapter
3 of \cite{GrotschelLS12}. This yields a polynomial time algorithm as it is
proved in \cite{GrotschelLS12}. We now explain in more detail each step of the
algorithm.
\begin{enumerate}
  \item The binary search over $\tau$ is the usual procedure to reduce the
  minimization of the norm $\norm{\vec{z} - \vec{w}}_2$ to satifiability
  queries of a set of convex constraints, in our case
  $\vec{z} \in \Domain_{r, B}$ and $\norm{\vec{z} - \vec{w}}_2 \le \tau$.
  \item The fact that the constraint $\norm{\vec{z} - \vec{w}}_2 \le \tau$ is
  satisfied through the execution of the ellipsoid algorithm is guaranteed
  because of the selection of the initial ellipsoid to be
  \[ \mathcal{E}_0 = \{\vec{z} \mid \norm{\vec{z} - \vec{w}}_2 \le \tau\}. \]
  \item The main technical difficulty is how to find a separating hyperplane
  between a vector $\vec{u}$ that is outside the set $\Domain_{r, B}$ and the
  convex set $\Domain_{r, B}$. First observe that $\vec{z} \in \Domain_{r, B}$
  is equivalent with $\norm{\vec{z}}_2 \le B$ and the following set of
  constraints
  \begin{align} \label{eq:intersectionOfEllipsoidsA}
    \sum_{i = 1}^n \left(y^{(i)} - \vec{z}^T \vec{x}^{(i)}\right)^2 \left( \vec{v}^T \vec{x}^{(i)} \right)^2 \le r  ~~~~ \forall \vec{v} \in \reals^k, ~ \norm{\vec{v}}_2 = 1
  \end{align}
  \noindent which after of simple calculations is equivalent with
  \begin{align} \label{eq:intersectionOfEllipsoids}
    \vec{z}^T \matr{P}_{\vec{v}} \vec{z} + \vec{q}_{\vec{v}}^T \vec{z} + s_{\vec{v}} \le r  ~~~~ \forall \vec{v} \in \reals^k, ~ \norm{\vec{v}}_2 = 1
  \end{align}
  \noindent with
  \begin{align*}
    \matr{P}_{\vec{v}} & = \sum_{i = 1}^n \left(\vec{v}^T \vec{x}^{(i)} \right)^2 \vec{x}^{(i)} \vec{x}^{(i) T}, \\
    \vec{q}_{\vec{v}} & = - 2 \sum_{i = 1}^n y^{(i)} \left(\vec{v}^T \vec{x}^{(i)} \right)^2 \vec{x}^{(i)}  \\
    \text{ and } ~~~~ s_{\vec{v}} & = \sum_{i = 1}^n \left(y^{(i)}\right)^2 \left(\vec{v}^T \vec{x}^{(i)} \right)^2.
  \end{align*}
  \noindent It is clear from the definition that
  $\matr{P}_{\vec{v}} \succeq \matr{0}$. Also observe that for any vector
  $\vec{z}$ such that $\vec{z}^T \matr{P}_{\vec{v}} \vec{z} = 0$ it also holds
  that $\vec{q}_{\vec{v}}^T \vec{z} = 0$. This holds because
  $\vec{z}^T \matr{P}_{\vec{v}} \vec{z} = 0$ is a sum of squares, hence it is
  zero if and only if all the terms are zero and
  $\vec{q}_{\vec{v}}^T \vec{z} = 0$ is a linear combination of these terms.
  Hence for any unit vector $\vec{v}$ the equivalent inequalities
  \eqref{eq:intersectionOfEllipsoidsA}, \eqref{eq:intersectionOfEllipsoids}
  describe an ellipsoid with its interior. Also the constraint
  $\norm{\vec{z}}_2 \le B$ can also be described as an ellipsoid constraint
  hence we may add this to the previous set of ellipsoid constraints that we
  want to satisfy.

  Let us assume that $\vec{u}$ violates some of the ellipsoid inequalities
  \eqref{eq:intersectionOfEllipsoids}. To find such one of the ellipsoids that
  does not contain $\vec{u}$ it suffices to compute the eigenvector $\vec{v}_m$
  that corresponds to the maximum eigenvalue of the following matrix
  \[ \matr{A} = \sum_{i=1}^k\left(y^{(i)} - \vec{u}^T \vec{x}^{(i)}\right)^2 \vec{x_i}^{(i)} \vec{x_i}^{(i) T}. \]
  If $\lambda_{\max}(\matr{A}) \le r$ then $\vec{u}$ satisfies all the
  ellipsoid constraints, otherwise it holds that
  \begin{align} \label{eq:defG}
    g \triangleq \vec{v}_m^T \left( \sum_{i=1}^k\left(y^{(i)} - \vec{u}^T \vec{x}^{(i)}\right)^2 \vec{x_i}^{(i)} \vec{x_i}^{(i) T} \right) \vec{v}_m > r.
  \end{align}
  This implies that $\vec{u}$ is outside the ellipsoid
  \[ \mathcal{E} = \left\{ \vec{z} \mid \vec{z}^T \matr{P}_{\vec{v}_m} \vec{z} + \vec{q}_{\vec{v}_m}^T \vec{z} + s_{\vec{v}_m} \le r \right\} \]
  \noindent and also from the definition of $\Domain_{r, B}$ we have
  $\Domain_{r, B} \subseteq \mathcal{E}$. Hence it suffices to find a hyperplane
  that separates $\vec{u}$ with $\mathcal{E}$. This is an easy task since we can
  define ellipsoid surface $\mathcal{S}$ that is parallel to $\mathcal{E}$ and
  passes through $\vec{u}$ as follows
  \[ \mathcal{S} = \left\{ \vec{z} \mid \vec{z}^T \matr{P}_{\vec{v}_m} \vec{z} + \vec{q}_{\vec{v}_m}^T + s_{\vec{v}_m} = g \right\} \]
  \noindent where $g$ is defined in \eqref{eq:defG} and the tangent hyperplane
  of $\mathcal{S}$ at $\vec{u}$ is a separating hyperplane between $\vec{u}$
  and $\mathcal{E}$. To compute the tangent hyperplane we can compute the
  gradient $\vec{d} = \nabla_{\vec{z}} \left( \vec{z}^T \matr{P}_{\vec{v}_m} \vec{z} + \vec{q}_{\vec{v}_m}^T + s_{\vec{v}_m} \right) \mid_{\vec{z} = \vec{u}}$ and define the following hyperplane
  \begin{align} \label{eq:hyperplaneDefinition}
    \mathcal{H} = \{\vec{z} \mid \vec{d}^T \vec{z} = \vec{d}^T \vec{u} \}.
  \end{align}
\end{enumerate}

\subsection{Bounded Step Variance and Strong Convexity} \label{sec:statistical}

  We are now ready to prove the (i) bounded variance and (ii) strong convexity
properties as descripted in the beginning of Section \ref{sec:estimation}. The
results are summarized in the following Theorem \ref{thm:conditionsSatisfied}.
The proof of Theorem \ref{thm:conditionsSatisfied} is presented in the Appendix
\ref{sec:thm:conditionsSatisfiedProof}.

\begin{theorem} \label{thm:conditionsSatisfied}
    Let $(\vec{x}^{(1)}, y^{(1)}), \dots, (\vec{x}^{(n)}, y^{(n)})$ be $n$
  samples from the linear regression model
  \eqref{eq:truncatedLinearRegressionDefinition} with parameter vector
  $\vec{w}^*$. If Assumptions \ref{asp:survivalProbability} and
  \ref{asp:logkAssumption} and $\norm{\vec{x}^{(i)}}_2 \le 1$ hold, then for
  every $\vec{w} \in \Domain_{r^*, \bar{B}}$ it
  holds that
  \begin{align} \label{eq:thm:condition1}
    \Exp \left[ \norm{\vec{v}^{(i)}}_2^2 \right] \le O(\poly(1/a) \cdot \bar{B}^2),
  \end{align}
        \begin{align} \label{eq:thm:condition2}
    \mathbf{H}_{-\bar{\ell}}(\vec{w}) = \frac{1}{n} \sum_{i = 1}^n \Exp \left[ \left( z^{(i)} - \Exp\left[ z^{(i)} \right] \right)^2 \vec{x}^{(i)} \vec{x}^{(i) T} \right] \succeq e^{-16} \cdot a^{10} \cdot \frac{1}{b^2} \cdot \matr{I},
  \end{align}
  \noindent where $y^{(i)} \sim \normal(\vec{w}^{* T} \vec{x}^{(i)}, 1, S)$,
  $z^{(i)} \sim \normal(\vec{w}^T \vec{x}^{(i)}, 1, S)$, $\vec{v}^{(i)}$ is the
  unbiased estimate of $\nabla f_i(\vec{w})$ according to Algorithm
  \ref{alg:projectedSGD}, and $r^* = 4 \log \left( \nicefrac{2}{a} \right) + 7$.
\end{theorem}

\subsection{Feasibility of Optimal Solution} \label{sec:optimal}

  As described in the high level description of our proof in the beginning of
the section, in order to be able to use strong convexity to prove the closeness
in parameter space of our estimator, we have to prove that
$\vec{w}^* \in \Domain_{r^*, \bar{B}}$. This is also needed to prove that all
the points $\vec{w} \in \Domain_{r^*, \bar{B}}$ satisfy the Assumption
\ref{asp:survivalProbability}, which we have used to prove the bounded variance
and the strong convexity property in Section \ref{sec:statistical}. The proof
of the following Lemma can be found in the Appendix \ref{sec:thirdLemmaProof}.

\begin{lemma} \label{lem:thirdLemma}
    Let $(\vec{x}^{(1)}, y^{(1)}), \dots, (\vec{x}^{(n)}, y^{(n)})$ be $n$
  samples from the linear regression model
  \eqref{eq:truncatedLinearRegressionDefinition} with parameters $\vec{w}^*$. If
  Assumptions \ref{asp:survivalProbability} and \ref{asp:logkAssumption} hold
  and $\norm{\vec{w}^*}_2 \le \bar{B}$ then,
  $\Prob \left( \vec{w}^* \in \Domain_{r^*, \bar{B}} \right) \ge 2/3$.
\end{lemma}

\subsection{Bounded Parameters} \label{sec:boundedParameters}

  Now we provide the necessary guarantees for the upper bound on the diameter
of $\Domain_{r^*, \bar{B}}$ and the absolute values of the coefficients $c_i$
that are needed to apply Theorem \ref{thm:mainSGDAnalysis}.

\begin{lemma} \label{lem:boundedParameters}
    The diameter of $\Domain_{r^*, \bar{B}}$ is at most $2 \bar{B}$. If we also
  assume that $\norm{\vec{x}^{(i)}}_2 \le 1$ then we have that
  \[ \abs{\Exp_{y \sim \normal(\vec{w}^{*T} \vec{x}^{(i)}, 1, S)}\left[y \right]} \le \bar{B} + \poly(\log(1/a)). \]
\end{lemma}

\begin{proof}
    The bound on the diameter of the set $\Domain_{r^*, \bar{B}}$ follows
  directly from its definition. For the coefficients $c_i$ we have that
  $\abs{c_i} = \abs{\Exp_{y \sim \normal(\vec{w}^{*T} \vec{x}^{(i)}, 1, S)}\left[y \right]}$
  we have from Lemma 6 of \cite{DaskalakisGTZ18} that
  \[ \abs{\Exp_{y \sim \normal(\vec{w}^{*T} \vec{x}^{(i)}, 1, S)}\left[y \right] - \vec{w}^{* T} \vec{x}^{(i)}} \le \sqrt{\log(1/a)} \]
  \noindent and hence using the assumption that $\norm{\vec{x}^{(i)}}_2 \le 1$
  we get that
  \begin{align} \label{eq:boundedParametersCoefficientBound}
    \abs{c_i} & \le \norm{\vec{w}^{*}}_2 \norm{\vec{x}^{(i)}}_2 + \sqrt{\log(1/a)} \le \bar{B} + + \sqrt{\log(1/a)}.
  \end{align}
  \noindent and the lemma follows.
\end{proof}

\subsection{Proof of Theorem \ref{thm:estimationTheorem}}
\label{sec:mainResultProof}

  In this section we analyze Algorithm \ref{fig:algorithm}, which implements
the projected stochastic gradient descent without replacement on the negative
log-likelihood landscape. Before running Algorithm \ref{fig:algorithm}, as we
explained before \eqref{eq:normalizedAssumptions} we apply a normalization so
that \eqref{eq:normalizedAssumptions}. Now we can use of the Lemmas and Theorems
that we proved in the previous sections with
$\bar{B} = B \cdot C \cdot \sqrt{k}$.

  First we observe that from Section \ref{sec:gradient} and
\ref{sec:projectionAlgorithmProof} we have a proof that we can: (a) efficiently
find one initial feasible point, (b) we can compute efficiently an unbiased
estimation of the gradient and (c) we can efficiently project to the set
$\Domain_{r^*, \bar{B}}$. These results prove that the Algorithm
\ref{fig:algorithm} that we analyze has running time that is polynomial in the
number of steps $M$ of stochastic gradient descent, and in the dimension $k$.

  It remains to upper bound the number of steps that the PSGD algorithm needs
to compute an estimate that is close in parameter space with $\vec{w}^*$. The
analysis of this estimation algorithm will be based on Theorem
\ref{thm:conditionsSatisfied} combined with the theorem about the performance
of the projected stochastic gradient descent without replacement
(Theorem \ref{thm:mainSGDAnalysis}). This is summarized in the following Lemma
whose proof follows from the lemmas that we presented in the previous sections
together with Theorem \ref{thm:mainSGDAnalysis}.

\begin{lemma} \label{lem:expectedPerformanceSGD}
    Let $\vec{w}^*$ be the underlying parameters of our model, let
  $f(\vec{w}) = \frac{1}{n} \sum_{i = 1}^n f_i(\vec{w}) = - \bar{\ell}(\vec{w})$
  where
  \begin{enumerate}
    \item[$\triangleright$] $f_i(\vec{w}) = c_i \cdot \vec{w}^T \vec{x}^{(i)} + q(\vec{w})$,
    \item[$\triangleright$] $c_i = \Exp_{y \sim \normal(\vec{w}^{*T} \vec{x}^{(i)}, 1, S)}\left[y \right]$,
    \item[$\triangleright$] $q(\vec{w}) = \frac{1}{n} \sum_{i = 1}^n \log \left( \int_S \exp \left(- \frac{1}{2} z^2 + z \cdot \vec{w}^T \vec{x}^{(i)} \right) dz \right)$.
  \end{enumerate}
  If for all $i \in [n]$, $\norm{\vec{x}^{(i)}}_2 \le 1$, then it holds that
  \[ \Exp\left[ f(\vec{\bar{w}}) \right] - f(\vec{w}^*) \le \poly \left(\frac{1}{a} \right) \cdot \frac{B^2}{b^2} \cdot \frac{1}{n} \left( 1 + \log(n) \right) \]
  where $\bar{\vec{w}}$ is the output of Algorithm \ref{alg:projectedSGD} after
  $n$ steps.
\end{lemma}

\begin{proof}
    Due to the first point of Theorem \ref{thm:conditionsSatisfied} we have
  that the condition (i) of Theorem \ref{thm:mainSGDAnalysis} is satisfied with
  $\rho = \poly(1/a) \cdot \bar{B}$. Also it is not hard to see that
  the Hessian $\mathbf{H}_{-\bar{\ell}}$ is equal to the Hessian
  $\mathbf{H}_{q}$. Therefore from the second point of Theorem
  \ref{thm:conditionsSatisfied} we have that the condition (ii) of Theorem
  \ref{thm:mainSGDAnalysis} is satisfied with $\lambda = \poly(a) / b^2$.
  Finally from Lemma \ref{lem:boundedParameters} we have that the condition
  (iii) of Theorem \ref{thm:mainSGDAnalysis} is satisfied. Hence we can apply
  Theorem \ref{thm:mainSGDAnalysis} and the lemma follows.
\end{proof}

\begin{prevproof}{Theorem}{thm:estimationTheorem}
    Applying Markov's inequality to Lemma \ref{lem:expectedPerformanceSGD} we
  get that
  \begin{equation} \label{eq:MarkovFinalBound}
    \Prob \left( f(\vec{\bar{w}}) - f(\vec{w}^*) \ge \poly \left( \frac{1}{\delta a} \right) \cdot \frac{\bar{B}^2}{b^2} \cdot \frac{k}{n} \left( 1 + \log(n) \right) \right) \le \delta.
  \end{equation}

  \noindent Hence we can condition on the event
  $f(\vec{\bar{w}}) - f(\vec{w}^*) \le \poly \left( \frac{1}{ \delta a } \right) \cdot \frac{\bar{B}^2}{b^2} \cdot \frac{k}{n} \left( 1 + \log(n) \right)$
  and we only loose probability at most $\delta$. Now we can use Theorem
  \ref{thm:conditionsSatisfied} which establishes the $\lambda$-strong
  convexity of $f$, for $\lambda = \poly(\alpha)/b^2$, which implies that
  $f(\vec{w}) - f(\vec{w}^*) \ge \frac{\lambda}{2} \norm{\vec{w} - \vec{w}^*}_2^2$,
  to get that
  \begin{equation} \label{eq:finalProofSecondToLast}
    \norm{\bar{\vec{w}} - \vec{w}^*}_2^2 \le \poly \left(\frac{1}{a} \right) \cdot \frac{\bar{B}^2}{b^4} \cdot \frac{k}{n} \left( 1 + \log(n) \right),
  \end{equation}
  \noindent and our theorem follows if we replace
  $\bar{B} = B \cdot C \cdot \sqrt{k}$ as we explained in
  \eqref{eq:normalizedAssumptions}.
\end{prevproof}

\subsection{Full Description of the Algorithm} \label{sec:algos}
\vspace{-10pt}
  \begin{algorithm}[H]
  \caption{Projected Stochastic Gradient Descent. Given access to samples from $\normal(\vec{w}^T \vec{x},1, S)$.}
  \begin{algorithmic}[1]
    \Procedure{Sgd}{$M, \lambda$}\Comment{\textit{$M$: number of steps, $\lambda$: parameter}}
        \State Apply a random permutation of the set of samples $\{(\vec{x}^{(1)}, y^{(1)}), \ldots, (\vec{x}^{(n)}, y^{(n)}))\}$
    \State $\vec{w}^{(0)} \gets$ \Call{ProjectToDomain}{$r^*, B, \vec{0}$}
    \For{$i = 1, \dots, M$}
      \State $\eta_i \gets \frac{1}{\lambda \cdot i}$
      \State $\vec{v}^{(i)} \gets \Call{GradientEstimation}{\vec{x}^{(i)}, \vec{w}^{(i - 1)}, y^{(i)}}$
      \State $\vec{r}^{(i)} \gets \vec{w}^{(i - 1)} - \eta_i \vec{v}^{(i)}$
      \State $\vec{w}^{(i)} \gets \Call{ProjectToDomain}{r^*, B, \vec{r}^{(i)}}$
    \EndFor    \State \textbf{return} $\bar{\vec{w}} \gets \frac{1}{M} \sum_{i = 1}^M \vec{w}^{(i)}$\Comment{\textit{output the average}}
  \EndProcedure
  \end{algorithmic}
  \label{alg:projectedSGD}
  \end{algorithm}
  \vspace{-15pt}
  \begin{algorithm}[H]
  \caption{The function to estimate the gradient of the function $f_i$ as in \eqref{eq:f_iGradient}.}
  \label{alg:estimateGradient}
  \begin{algorithmic}[1]
    \Function{GradientEstimation}{$\vec{r}, \vec{w},y$}
    \State pick $\vec{x}$ at random from the set $\left\{\vec{x}^{(1)}, \dots, \vec{x}^{(n)}\right\}$
    \State Sample $z$ from $\normal(\vec{w}^T \vec{x}, 1, S)$ according to the Appendix \ref{sec:unionIntervalsSampling}
    \State \textbf{return} $y \vec{r} - z \vec{x}$
    \EndFunction
  \end{algorithmic}
  \label{alg:rejectionSampling}
  \end{algorithm}
  \vspace{-15pt}
  \begin{algorithm}[H]
  \caption{The function that projects a current guess back to the domain $\Domain_{r, B}$ (see Section \ref{sec:projectionAlgorithmProof}).}
  \label{alg:projectToDomain}
  \begin{algorithmic}[1]
    \Function{ProjectToDomain}{$r, B, \vec{w}$} \Comment{\textit{$r, B$ are the parameters of the domain $\Domain_{r, B}$}}
                    \State $\hat{\tau} \gets \arg \min_{\tau} \{\textsc{Ellipsoid($\vec{w}$, $\tau$, $r$)} \neq \text{``Empty''}\}$ \Comment{\textit{find $\tau$ via binary search}}
    \State \textbf{return} $\textsc{Ellipsoid($\vec{w}$, $\hat{\tau}$, $r$)}$
        \EndFunction
    \vspace{5pt}
    \Function{Ellipsoid}{$\vec{w}$, $\tau$, $r$, $B$} \Comment{\textit{return a point $\vec{z}$ in $\Domain_{r, B}$ with $\norm{\vec{z} - \vec{w}}_2 \le \tau$ or ``Empty''}}
    \State $\mathcal{E}_0 \gets \{\vec{z} \mid \norm{\vec{z} - \vec{w}}_2 \le \tau\}$
    \State \vspace{-24pt}
    \begin{align*}
      \textbf{return} ~~ & \text{the result of the ellipsoid method with initial ellipsoid $\mathcal{E}_0$ and $\textsc{FindSeparation}$} \\[-3pt]
      & \text{as a separation oracle}
    \end{align*}
    \vspace{-24pt}
                \EndFunction
    \vspace{5pt}
    \Function{FindSeparation}{$\vec{u}$, $r$, $B$} \Comment{\textit{find a separating hyperplane between $\vec{u}$ and $\Domain_{r, B}$}}
    \State $\matr{A} \gets \sum_{i=1}^n \left(y^{(i)} - \vec{u}^T \vec{x}^{(i)}\right)^2 \vec{x_i}^{(i)} \vec{x_i}^{(i) T}$
    \If{$\lambda_{\max}(\matr{A}) \leq r$}
      \If{$\norm{\vec{u}}_2 \le B$}
        \State \textbf{return} ``is member''
      \Else
        \State \textbf{return} a separating hyperplane between the vector
         $\vec{u}$ and the ball with radius $B$
      \EndIf
    \Else
    \State $\vec{v}_m \gets$ eigenvector of $\matr{A}$ that corresponds to the maximum eigenvalue $\lambda_{\max}(\matr{A})$
    \State $\mathcal{E} \gets \left\{ \vec{z} \in \reals^k \mid \sum_{i=1}^n \left(y^{(i)} - \vec{z}^T \vec{x}^{(i)}\right)^2 \left(\vec{v}_m^T \vec{x}^{(i)} \right)^2 \le r \right\}$
    \State \textbf{return} a separating hyperplane between the vector $\vec{u}$ and the ellipsoid $\mathcal{E}$ (See \eqref{eq:hyperplaneDefinition})
    \EndIf
    \EndFunction
  \end{algorithmic}
  \label{alg:projectionAlgorithm}
  \end{algorithm}
  \vspace{-10pt}
  \begin{figure}[H]
    \caption{Description of the Stochastic Gradient Descent (SGD) algorithm
    without replacement for estimating the parameters of a truncated linear
    regression.}
    \label{fig:algorithm}
  \end{figure}

\section{Learning 1-Layer Neural Networks with Noisy Activation Function} \label{sec:neural}

  In this section we will describe how we can use our truncated regression
algorithm to provably learn the parameters of an one layer neural network
with noisy activation functions. Noisy activation function have been explored
by \cite{NairH10}, \cite{BengioLC13} and \cite{GulcehreMDB16} as we have
discussed in the introduction. The problem of estimating the parameters of
such a neural network is a challenging problem and no theoretically rigorous
methods are known. In this section we show that this problem can be
formulated as a truncated linear regression problem which we can efficiently
solve using our Algorithm \ref{alg:projectedSGD}.

  Let $g : \reals \to \reals$ be a random map that corresponds to a noisy
rectifier linear unit, i.e. $g(x) = \max\{0, x + \eps\}$ where $\eps$ is a
standard normal random variable. Then an one layer neural network with noisy
activation functions is the multivalued function $f$ parameterized by the
vector $\vec{w} \in \reals^k$ such that
$f_{\vec{w}}(\vec{x}) = g(\vec{w}^T \vec{x})$. In the realizable, supervised
setting we observe $n$ labeled samples of the form
$\left(\vec{x}^{(i)}, y^{(i)}\right)$ and we want to estimate the parameters
$\matr{W}$ that better capture the samples we have observed. We remind that the
assumption that $\left(\vec{x}^{(i)}, y^{(i)}\right)$ is realizable means that
there exists a $\vec{w}^*$ such that for all $i$ it holds
$y^{(i)} = f_{\vec{w}^*} \left(\vec{x}^{(i)}\right) = g(\vec{w}^{* T} \vec{x})$.

  Our SGD algorithm then gives a rigorous method to estimate $\matr{W}^*$ if we
assume that the inputs $\vec{x}^{(i)}$ together with the truncation of the
activation function satisfy Assumption \ref{asp:survivalProbability} and
Assumption \ref{asp:logkAssumption}. Using Theorem \ref{thm:estimationTheorem}
we can then bound the number of samples that we need for this learning task.
These results are summarized in the following corollary which directly follows
from Theorem \ref{thm:estimationTheorem}.

\begin{corollary} \label{cor:neural}
    Let $\left(\vec{x}^{(i)}, y^{(i)}\right)$ be $n$ i.i.d. samples drawn
  according to the following distribution
  \[ y^{(i)} = f_{\vec{w}^*}\left(\vec{x}^{(i)}\right)
     = \max\{0, \vec{w}^{* T} \vec{x}^{(i)} + \eps^{(i)}\} \]
  \noindent with $\eps^{(i)} \sim \normal(0, 1)$. Assume also that $\left(\vec{x}^{(i)}, y^{(i)}\right)$ and $\vec{w}^*$ satisfy Assumption
  \ref{asp:survivalProbability} and \ref{asp:logkAssumption}, and that $\norm{\vec{x}^{(i)}}_\infty \le B$ and $\norm{\vec{w}^*}_2 \le C$. Then
  the SGD Algorithm \ref{alg:projectedSGD} outputs an estimate $\hat{\matr{w}}$
  such that
  \[ \norm{\hat{\vec{w}} - \vec{w}^*}_{2} \le \poly\left(\frac{1}{a} \right) \cdot \frac{B \cdot C}{b^2} \cdot \sqrt{\frac{k}{n} \log n} \]
  \noindent with probability at least $2/3$. Moreover if $S$ is a union of $r$
  subintervals then Algorithm \ref{alg:projectedSGD} runs in polynomial time.
\end{corollary}

\noindent Note that the aforementioned problem is easier than the problem that
we solve in Section \ref{sec:estimation}. The reason is that in the neural
network setting even the samples $y^{(i)}$ that are filtered by the activation
function are available to us and hence we have the additional information that
we can compute their percentage. In Corollary \ref{cor:neural} we don't use at
all this information.

\section{Experiments} \label{sec:experiments}

  To validate the performance of our proposed algorithm we constructed a
synthetic dataset with various datapoints $\vec x_i \in \mathbb{R}^{10}$ that
were drawn uniformly at random from a Gaussian Distribution
$\normal(\matr{0},\matr{I})$. For each of these datapoints, we generated the
corresponding $y_i$ as $y_i = \vec w^T \vec x_i + \eps_i$, where $\eps_i$ where
drawn independently from $\normal(0,1)$ and $w$ was chosen to be the all-ones
vector $\vec 1$. We filtered the dataset to keep all samples with $y_i > 4$ and
$\vec w^T \vec x < 2$. We note that to run the projection step of our algorithm
we use the convex optimization library \texttt{cvxpy}.

Figure~\ref{fig:comparison-with-ols} shows the comparison with ordinary least squares. You can see that even though the OLS estimator quickly converges, its estimate is biased due to the data truncation. As a result, the estimates produced tend to be significantly larger in magnitude than the true $\vec  w = \vec 1$. In contrast, our proposed method is able to correct for this bias achieving an estimate that improves with the number of samples $n$ at an optimal rate of $1/\sqrt{n}$, despite the adversarial nature of the filtering that kept only significantly high values of $y$.

\begin{figure}[H]
\centering
  \includegraphics[width=0.6\textwidth]{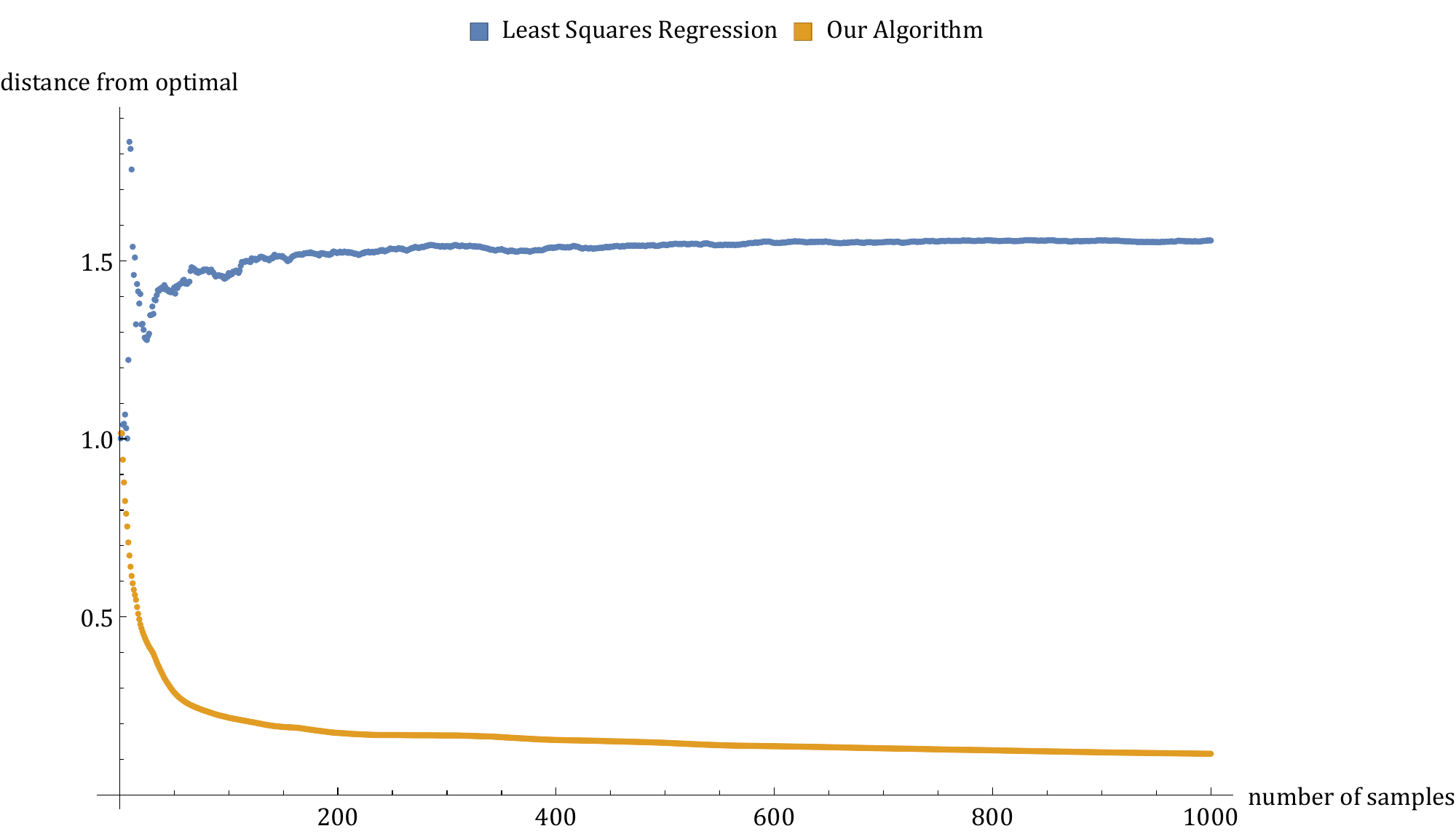}
  \caption{Comparison of the proposed method with ordinary least squares.}
  \label{fig:comparison-with-ols}
\end{figure}
 \bibliographystyle{alpha}
\bibliography{ref}

\newcommand{\etalchar}[1]{$^{#1}$}
\begin{thebibliography}{CLMW11}

\bibitem[Ame73]{amemiya1973regression}
Takeshi Amemiya.
\newblock Regression analysis when the dependent variable is truncated normal.
\newblock {\em Econometrica: Journal of the Econometric Society}, pages
  997--1016, 1973.

\bibitem[BC14]{BalakrishnanCramer}
N~Balakrishnan and Erhard Cramer.
\newblock {\em The art of progressive censoring}.
\newblock Springer, 2014.

\bibitem[BDLS17]{BDLS17}
Sivaraman Balakrishnan, Simon~S. Du, Jerry Li, and Aarti Singh.
\newblock Computationally efficient robust sparse estimation in high
  dimensions.
\newblock In {\em Proceedings of the 30th Conference on Learning Theory, {COLT}
  2017, Amsterdam, The Netherlands, 7-10 July 2017}, pages 169--212, 2017.

\bibitem[BJK15]{BJK15}
Kush Bhatia, Prateek Jain, and Purushottam Kar.
\newblock Robust regression via hard thresholding.
\newblock In {\em Advances in Neural Information Processing Systems}, pages
  721--729, 2015.

\bibitem[BLC13]{BengioLC13}
Yoshua Bengio, Nicholas L{\'e}onard, and Aaron Courville.
\newblock Estimating or propagating gradients through stochastic neurons for
  conditional computation.
\newblock {\em arXiv preprint arXiv:1308.3432}, 2013.

\bibitem[BLM13]{BoucheronLM13}
St{\'e}phane Boucheron, G{\'a}bor Lugosi, and Pascal Massart.
\newblock {\em Concentration inequalities: A nonasymptotic theory of
  independence}.
\newblock Oxford university press, 2013.

\bibitem[Bre96]{breen1996regression}
Richard Breen.
\newblock {\em Regression models: Censored, sample selected, or truncated
  data}, volume 111.
\newblock Sage, 1996.

\bibitem[Che12]{Chevillard12}
Sylvain Chevillard.
\newblock The functions erf and erfc computed with arbitrary precision and
  explicit error bounds.
\newblock {\em Information and Computation}, 216:72--95, 2012.

\bibitem[CLMW11]{CLMW11}
Emmanuel~J Cand{\`e}s, Xiaodong Li, Yi~Ma, and John Wright.
\newblock Robust principal component analysis?
\newblock {\em Journal of the ACM (JACM)}, 58(3):11, 2011.

\bibitem[Coh16]{Cohen91}
A~Clifford Cohen.
\newblock {\em Truncated and censored samples: theory and applications}.
\newblock CRC press, 2016.

\bibitem[CSV17]{CSV17}
Moses Charikar, Jacob Steinhardt, and Gregory Valiant.
\newblock Learning from untrusted data.
\newblock In {\em Proceedings of the 49th Annual {ACM} {SIGACT} Symposium on
  Theory of Computing, {STOC} 2017, Montreal, QC, Canada, June 19-23, 2017},
  pages 47--60, 2017.

\bibitem[DGTZ18]{DaskalakisGTZ18}
Constantinos Daskalakis, Themis Gouleakis, Christos Tzamos, and Manolis
  Zampetakis.
\newblock Efficient statistics, in high dimensions, from truncated samples.
\newblock In {\em the 59th Annual IEEE Symposium on Foundations of Computer
  Science (FOCS)}, 2018.

\bibitem[DKK{\etalchar{+}}16]{DKK+16b}
Ilias Diakonikolas, Gautam Kamath, Daniel~M. Kane, Jerry Li, Ankur Moitra, and
  Alistair Stewart.
\newblock Robust estimators in high dimensions without the computational
  intractability.
\newblock In {\em {IEEE} 57th Annual Symposium on Foundations of Computer
  Science, {FOCS} 2016, 9-11 October 2016, Hyatt Regency, New Brunswick, New
  Jersey, {USA}}, pages 655--664, 2016.

\bibitem[DKK{\etalchar{+}}17]{DKK+17}
Ilias Diakonikolas, Gautam Kamath, Daniel~M. Kane, Jerry Li, Ankur Moitra, and
  Alistair Stewart.
\newblock Being robust (in high dimensions) can be practical.
\newblock In {\em Proceedings of the 34th International Conference on Machine
  Learning, {ICML} 2017, Sydney, NSW, Australia, 6-11 August 2017}, pages
  999--1008, 2017.

\bibitem[DKK{\etalchar{+}}18]{DKK+18}
Ilias Diakonikolas, Gautam Kamath, Daniel~M. Kane, Jerry Li, Ankur Moitra, and
  Alistair Stewart.
\newblock Robustly learning a gaussian: Getting optimal error, efficiently.
\newblock In {\em Proceedings of the Twenty-Ninth Annual {ACM-SIAM} Symposium
  on Discrete Algorithms, {SODA} 2018, New Orleans, LA, USA, January 7-10,
  2018}, pages 2683--2702, 2018.

\bibitem[DKS19]{diakonikolas2018efficient}
Ilias Diakonikolas, Weihao Kong, and Alistair Stewart.
\newblock Efficient algorithms and lower bounds for robust linear regression.
\newblock In {\em the 30th Annual ACM-SIAM Symposium on Discrete Algorithms
  (SODA)}, 2019.

\bibitem[Fis31]{fisher31}
RA~Fisher.
\newblock Properties and applications of {Hh} functions.
\newblock {\em Mathematical tables}, 1:815--852, 1931.

\bibitem[Gal97]{Galton1897}
Francis Galton.
\newblock An examination into the registered speeds of american trotting
  horses, with remarks on their value as hereditary data.
\newblock {\em Proceedings of the Royal Society of London},
  62(379-387):310--315, 1897.

\bibitem[GLS12]{GrotschelLS12}
Martin Gr{\"o}tschel, L{\'a}szl{\'o} Lov{\'a}sz, and Alexander Schrijver.
\newblock {\em Geometric algorithms and combinatorial optimization}, volume~2.
\newblock Springer Science \& Business Media, 2012.

\bibitem[GMDB16]{GulcehreMDB16}
Caglar Gulcehre, Marcin Moczulski, Misha Denil, and Yoshua Bengio.
\newblock Noisy activation functions.
\newblock In {\em International Conference on Machine Learning}, pages
  3059--3068, 2016.

\bibitem[Gor41]{gordon41}
Robert~D Gordon.
\newblock Values of mills' ratio of area to bounding ordinate and of the normal
  probability integral for large values of the argument.
\newblock {\em The Annals of Mathematical Statistics}, 12(3):364--366, 1941.

\bibitem[HM13]{HM13}
Moritz Hardt and Ankur Moitra.
\newblock Algorithms and hardness for robust subspace recovery.
\newblock In {\em Conference on Learning Theory}, pages 354--375, 2013.

\bibitem[Hot48]{Hotelling48}
Harold Hotelling.
\newblock Fitting generalized truncated normal distributions.
\newblock In {\em Annals of Mathematical Statistics}, volume~19, pages
  596--596, 1948.

\bibitem[HW77]{hausman1977social}
Jerry~A Hausman and David~A Wise.
\newblock Social experimentation, truncated distributions, and efficient
  estimation.
\newblock {\em Econometrica: Journal of the Econometric Society}, pages
  919--938, 1977.

\bibitem[Li17]{Li17}
Jerry Li.
\newblock Robust sparse estimation tasks in high dimensions.
\newblock {\em CoRR}, abs/1702.05860, 2017.

\bibitem[LRV16]{LRV16}
Kevin~A. Lai, Anup~B. Rao, and Santosh Vempala.
\newblock Agnostic estimation of mean and covariance.
\newblock In {\em {IEEE} 57th Annual Symposium on Foundations of Computer
  Science, {FOCS} 2016, 9-11 October 2016, Hyatt Regency, New Brunswick, New
  Jersey, {USA}}, pages 665--674, 2016.

\bibitem[Mad86]{Maddala1986}
Gangadharrao~S Maddala.
\newblock {\em Limited-dependent and qualitative variables in econometrics}.
\newblock Cambridge university press, 1986.

\bibitem[NH10]{NairH10}
Vinod Nair and Geoffrey~E Hinton.
\newblock Rectified linear units improve restricted boltzmann machines.
\newblock In {\em Proceedings of the 27th international conference on machine
  learning (ICML-10)}, pages 807--814, 2010.

\bibitem[NJN19]{NagarajJN19}
Dheeraj Nagaraj, Prateek Jain, and Praneeth Netrapalli.
\newblock Sgd without replacement: Sharper rates for general smooth convex
  functions.
\newblock In {\em International Conference on Machine Learning}, pages
  4703--4711, 2019.

\bibitem[Pea02]{Pearson1902}
Karl Pearson.
\newblock On the systematic fitting of frequency curves.
\newblock {\em Biometrika}, 2:2--7, 1902.

\bibitem[PL08]{PearsonLee1908}
Karl Pearson and Alice Lee.
\newblock On the generalised probable error in multiple normal correlation.
\newblock {\em Biometrika}, 6(1):59--68, 1908.

\bibitem[PVRB18]{PillaudRB18}
Loucas Pillaud-Vivien, Alessandro Rudi, and Francis Bach.
\newblock Statistical optimality of stochastic gradient descent on hard
  learning problems through multiple passes.
\newblock In {\em Advances in Neural Information Processing Systems}, pages
  8114--8124, 2018.

\bibitem[Sch86]{Schneider86}
Helmut Schneider.
\newblock {\em Truncated and censored samples from normal populations}.
\newblock Marcel Dekker, Inc., 1986.

\bibitem[SCV18]{SCV18}
Jacob Steinhardt, Moses Charikar, and Gregory Valiant.
\newblock Resilience: {A} criterion for learning in the presence of arbitrary
  outliers.
\newblock In {\em 9th Innovations in Theoretical Computer Science Conference,
  {ITCS} 2018, January 11-14, 2018, Cambridge, MA, {USA}}, pages 45:1--45:21,
  2018.

\bibitem[Sha16]{Shamir16}
Ohad Shamir.
\newblock Without-replacement sampling for stochastic gradient methods.
\newblock In {\em Advances in neural information processing systems}, pages
  46--54, 2016.

\bibitem[Tob58]{tobin1958estimation}
James Tobin.
\newblock Estimation of relationships for limited dependent variables.
\newblock {\em Econometrica: journal of the Econometric Society}, pages 24--36,
  1958.

\bibitem[Tro12]{Tropp12}
Joel~A Tropp.
\newblock User-friendly tail bounds for sums of random matrices.
\newblock {\em Foundations of computational mathematics}, 12(4):389--434, 2012.

\bibitem[Tuk49]{Tukey49}
John~W Tukey.
\newblock Sufficiency, truncation and selection.
\newblock {\em The Annals of Mathematical Statistics}, pages 309--311, 1949.

\bibitem[XCM10]{XCM10}
Huan Xu, Constantine Caramanis, and Shie Mannor.
\newblock Principal component analysis with contaminated data: The high
  dimensional case.
\newblock {\em arXiv preprint arXiv:1002.4658}, 2010.

\end{thebibliography}

\appendix

\section{Proofs Omitted from the Main Text} \label{omitted}

  In this section we present all the formal proofs that are omitted from the
main text. We start with some concentration results and some auxiliary lemmas
that will be useful for the presentation of the missing proofs.

\subsection{Useful Concentration Results} \label{sec:app:concentration}
  The following lemma is useful in cases where one wants to show
concentration of a weighted sum of i.i.d \textit{sub-gamma} random variables.

\begin{definition}(\textsc{Sub-Gamma Random Variables})
    A random variable $x$ is called \textit{sub-gamma} random variable if it
  satisfies
  \[ \log \left( \Exp\left[ \exp(\lambda x) \right] \right) \le
       \frac{\lambda^2 v}{2(1 - c \lambda)} ~~~~ \forall \lambda \in
       \left[0, \nicefrac{1}{c}\right] \]
  for some positive constants $v, c$. We call $\Gamma_+(v, c)$ the set of all
  sub-gamma random variables.
\end{definition}

\begin{theorem}[Section 2.4 of \cite{BoucheronLM13}.] \label{thm:gammaConcentrationLemma}
    Let $x^{(1)}$, \dots, $x^{(n)}$ be i.i.d. random variables such that
  $x^{(i)} \in \Gamma_+(v, c)$, $\vec{a} \in \reals_+^d$. Then,
  the following inequalities hold for any $t \in \reals_+$.
  \[ \Prob \left( \sum_{i = 1}^d a_i \left(x^{(i)} - \Exp\left[x^{(i)}\right]\right)
       \ge \norm{\vec{a}}_2 \sqrt{2 v t} + \norm{\vec{a}}_{\infty} c t \right)
       \le \exp(-t), \]
  \[ \Prob \left( - \sum_{i = 1}^d a_i \left(x^{(i)} - \Exp\left[x^{(i)}\right]\right)
       \ge \norm{\vec{a}}_2 \sqrt{2 v t} + \norm{\vec{a}}_{\infty} c t \right)
       \le \exp(-t). \]
\end{theorem}

\noindent We also need a matrix concentration inequality analog to the
Bernstein inequality for real valued random variables. For a proof of this
inequality we refer to Section 6.2 of \cite{Tropp12}.

\begin{theorem}[Theorem 6.2 of \cite{Tropp12}] \label{thm:matrixBernstein}
    Consider a finite sequence $\{ \matr{Z}_i \}$ of independent, random,
  self-adjoint matrices with dimension $k$. Assume that
  \[ \Exp\left[\matr{Z}_i\right] = \matr{0}
      \quad\text{and}\quad
      \Exp\left[\matr{Z}_i^p\right] \preceq \frac{p!}{2} \cdot R^{p-2}
                \matr{A}_i^2
      \quad\text{for $p = 2, 3, 4, \dots$.} \]
  \noindent Compute the variance parameter
  \[ \sigma^2 := \norm{ \sum_i \matr{A}_i^2 }. \]
  Then the following chain of inequalities holds for all $t \geq 0$.
  \begin{align*}
    \Prob\left( \lambda_{\max}\left( \sum_i \matr{Z}_i \right) \geq t \right)
      & \leq k \cdot \exp \left( \frac{-t^2/2}{\sigma^2 + R t} \right)
        \tag{i} \\
      & \leq \begin{cases}
    	         k \cdot \exp( -t^2/4\sigma^2 ) &
                  \text{for $t \leq \sigma^2 / R$;} \\
    	         k \cdot \exp( -t/4R ) & \text{for $t \geq \sigma^2 / R$.}
             \end{cases} \tag{ii}
  \end{align*}
\end{theorem}

\subsection{Auxiliary Lemmas} \label{sec:app:aux}

\noindent The following lemma will be useful in the proof of lemma \ref{lem:thirdLemma}.

\begin{lemma} \label{lem:truncatedGaussianToExponential}
    Let $z$ be a random variable that follows a truncated Gaussian distribution
  $\normal(0, 1, S)$ with survival probability $a$ then there exists a real
  value $q > 0$ such that
  \begin{Enumerate}
    \item $q \le 2 \log\left( \nicefrac{2}{a} \right)$,
    \item the random variable $z^2 - q$ is stochastically dominated by a
          sub-gamma random variable $u \in \Gamma_+\left(1, 2\right)$ with
          $\Exp[u] = \nicefrac{1}{2}$.
  \end{Enumerate}
\end{lemma}

\begin{proof}
    We will prove that the random variable $z$ stochastically dominated by an
  exponential random variable. First observe that the distribution of $z^2$
  for different sets $S$ is stochastically dominated from the distribution pf
  $z^2$ when $S = S^* = S_q = \{ z \mid z^2 \ge q \}$, where q is chosen such
  that $\normal(0, 1; S^*) = a$. To prove this let $F_S$ be the cumulative
  distribution function of $z^2$ when the truncation set is $S$, we have that
  $\normal(0, 1; S) = \normal(0, 1; S^*) = a$ and hence
  \[ \Prob_{z \sim \normal(0, 1, S)} \left( z^2 \ge t \right) = \frac{1}{a}
     \normal(0, 1; S \cap S_t), \]
  \[ \Prob_{z \sim \normal(0, 1, S^*)} \left( z^2 \ge t \right) = \frac{1}{a}
     \normal(0, 1; S^* \cap S_t). \]
  \noindent We now prove that
  $\normal(0, 1; S \cap S_t) \le \normal(0, 1; S^* \cap S_t)$. If $t \ge q$
  then $S^* \cap S_t = S_t$ and hence $S \cap S_t \subseteq S^* \cap S_t$
  which implies
  $\normal(0, 1; S \cap S_t) \le \normal(0, 1; S^* \cap S_t)$. If $t \le q$
  then $S^* \cap S_t = S^*$ and hence
  \[ \normal(0, 1; S^* \cap S_t) = \normal(0, 1; S^*) = \normal(0, 1; S) \ge
     \normal(0, 1; S \cap S_t). \]
  \noindent Therefore
  $\normal(0, 1; S \cap S_t) \le \normal(0, 1; S^* \cap S_t)$ and this implies
  $F_S(t) \ge F_{S^*}(t)$, which implies that the distribution of $z^2$ for
  different sets $S$ is stochastically dominated from the distribution pf
  $z^2$ when $S = S^*$. Hence we can focus on the distribution of $z^2$ when
  $z \sim \normal(0, 1, S^*)$. First we have to get an upper bound on $q$. To
  do so we consider the $Q$-function of the standard normal distribution and
  we have that $a = \normal(0, 1; S^*) = 2 Q(\sqrt{q})$. But by Chernoff bound
  we have that $Q(\sqrt{q}) \le \exp\left( - \nicefrac{q}{2} \right)$ which implies
  \[ q \le 2 \log \left( \frac{2}{a} \right). \]
  \noindent Let $F_z$ the cumulative density function of $z$ and $F_{z^2}$
  the cumulative density function of $z^2$, we have that
  \[ F_{z^2}(t) = F_z(\sqrt{t}) - F_z(- \sqrt{t}), \]
  \noindent but we know that
  \[ F_z(t) = \frac{1}{a}
                \begin{cases}
                  \Phi(t)             & ~~~~~ t \le - \sqrt{q} \\
                  \Phi(- \sqrt{q})             & ~~~~~ - \sqrt{q} < t < \sqrt{q} \\
                  \Phi(t) - (\Phi(\sqrt{q}) - \Phi(-\sqrt{q}))  & ~~~~~ t \ge \sqrt{q}
                \end{cases} \]
  \noindent Hence we have that
  \[ F_{z^2}(t) = \frac{1}{a}
                \begin{cases}
                  0                                        & ~~~~~ t < q \\
                  \Phi(\sqrt{t}) - \Phi(-\sqrt{t}) - (\Phi(\sqrt{q}) - \Phi(-\sqrt{q})) & ~~~~~ t \ge q
                \end{cases} \]
\noindent But we know that $\Phi(\sqrt{t}) - \Phi(-\sqrt{t})$ is the
cumulative density function of the square of a Gaussian distribution, namely
is a gamma distribution with both parameters equal to $\nicefrac{1}{2}$. This
means that
$\Phi(\sqrt{t}) - \Phi(-\sqrt{t}) = \int_{0}^t
 \frac{\exp\left( -\nicefrac{\tau}{2} \right)}{\sqrt{2 \pi \tau}} d \tau$
and hence we get
\[ F_{z^2}(t) = \frac{1}{a}
              \begin{cases}
                0                                        & ~~~~~ t < q \\
                \int_{q}^t
                 \frac{\exp\left( -\nicefrac{\tau}{2} \right)}{\sqrt{2 \pi \tau}} d \tau & ~~~~~ t \ge q
              \end{cases} \]
\noindent If we define the random variable $v = z^2 - q$ then the cumulative
density function $F_v$ of $v$ is equal to
\[ F_{v}(t) = \frac{1}{a} \int_{0}^t
 \frac{\exp\left( -\nicefrac{\tau + q}{2} \right)}{\sqrt{2 \pi (\tau + q)}}
 d \tau \]
\noindent which implies that the probability density function $f_v$ of $v$ is
equal to
\[ f_v(t) = \frac{1}{a} \frac{\exp\left( -\nicefrac{(t + q)}{2} \right)}
   {\sqrt{2 \pi (t + q)}} d \tau. \]
\noindent It is easy to see that the density of $f_v$ is stochastically
dominated from the density of the exponential distribution
$g(t) = \frac{1}{2} \exp\left(- \frac{x}{2}\right)$. The reason is that
$f_v(t)$ and $g(t)$ are single crossing and $g(t)$ dominates when
$t \to \infty$. Then it is easy to see that for the cumulative it holds that
$G(t) \le F_v(t)$. Hence $G(t)$ stochastically dominates $F_v(t)$. Finally we
have that $G(t)$ is a sub-gamma $\Gamma_+(1, 2)$ and hence $v$ is also
sub-gamma $\Gamma_+(1, 2)$ and the claim follows.
\end{proof}

The following lemma lower bounds the variance of $z \sim \normal(\vec{w}^T \vec{x}, 1, S)$, and will be useful for showing strong convexity of the log likelihood for all values of the parameters in the projection set.

\begin{lemma} \label{lem:fourthLemma}
    Let $\vec{x} \in \reals^k$, $\vec{w} \in \reals^{k}$, and
  $z \sim \normal(\vec{w}^T \vec{x}, 1, S)$ then
  \[ \Exp \left[ \left( z - \Exp\left[ z \right] \right)^2 \right]
      \ge \frac{\alpha(\vec{w}, \vec{x})^2}{12}.  \]
\end{lemma}

\begin{proof}  We want to bound the following expectation:
  $\lambda_1 = \Exp_{\vec{x}\sim \normal(\vec{0},\matr{I},S)}[(x_1 - \mu_{S,1})^2]= \Var_{\vec{x}\sim \normal(\vec{0},\matr{I},S)}[g_1] $, where $g_1$ denotes the marginal distribution of $g$ along the direction of $\vec{e_1}$. Since
  $\normal(\vec{0},\matr{I};S) = \alpha$, the worst case set (i.e the one that minimizes $\Var[g_1]$) is the one that has $\alpha$ mass as close as possible to the hyperplane $x_1 = \mu_{S,1}$. However, the maximum mass that a gaussian places at the set $\{x_1:|x_1-\mu_{S,1}|<c\}$ is at most $2c$
  as the density of the univariate gaussian is at most $1$. Thus the $\Exp_{\vec{x}\sim \normal(\vec{0},\matr{I},S)}[(x_1 - \mu_{S,1})^2]$ is at least
  the variance of the uniform distribution $U[-\alpha/2,\alpha/2]$ which
  is $\alpha^2/12$. Thus $\lambda_i \ge \lambda_1 \ge \alpha^2/12$.
\end{proof}

\subsection{Proof of Lemma \ref{lem:firstLemma}} \label{sec:firstLemmaProof}

    We have that
  \begin{align*}
    \alpha(\vec{w}, \vec{x})
      & = \Exp_{y \sim \normal(\vec{w}^T \vec{x}, \matr{I})}[ \chara_{y \in S} ] \\
      & = \Exp_{y \sim \normal(\vec{w}'^T \vec{x}, \matr{I})}[ \chara_{y \in S} \exp(\norm{ y - \vec{w}'^T x }^2/2 - \norm{ y - \vec{w}^T x }^2/2 ) ] \\
      & = \alpha(\vec{w}', \vec{x}) \cdot \Exp_{y \sim \normal(\vec{w}'^T \vec{x}, 1, S)}[ \exp(\frac12 \norm{ y - \vec{w}'^T x }^2 - \frac12 \norm{ y - \vec{w}^T x }^2 ) ] \\
      & \ge \alpha(\vec{w}', \vec{x}) \cdot  \exp( \Exp_{y \sim \normal(\vec{w}'^T \vec{x}, 1, S)}[ \frac12 \norm{ y - \vec{w}'^T x }^2 - \frac12 \norm{ y - \vec{w}^T x }^2 ] ) \\
      & \ge \alpha(\vec{w}', \vec{x}) \cdot  \exp( \Exp_{y \sim \normal(\vec{w}'^T \vec{x}, 1, S)}[ - \frac12 \norm{ y - \vec{w}'^T x }^2 - \abs{\left( \vec{w} - \vec{w}' \right)^T \vec{x}}^2 ] )   \\
      & \ge \exp\left(- \abs{\left( \vec{w} - \vec{w}' \right)^T \vec{x}}^2 - 2 \log \left(\frac{1}{\alpha(\vec{w}', \vec{x})}\right) - 2\right).
  \end{align*}
  which implies the desired bound. The first inequality follows from Jensen's
  inequality. The second follows from the fact that
  $\norm{u - v}^2 \le 2\norm{u}^2 + 2\norm{v}^2$. The final inequality follows
  from Lemma~\ref{lem:secondLemma}. The same way we can prove the second part of
  the lemma.

  \subsection{Proof of Lemma \ref{lem:secondLemma}} \label{sec:secondLemmaProof}

    Let $\alpha(\vec{w}, \vec{x}) = c$ for some fixed constant $c<1$. Note
  that, $\normal(\vec{w}^T \vec{x}, 1, S)$ is a truncated version of the normal
  distribution $\normal(\vec{w}^T \vec{x},1)$. Assume, without loss of
  generality that $\vec{w}^T \vec{x} = 0$. Our goal is to upper bound the
  second moment of this distribution around $0$. It is clear that for this
  moment to be maximized, we need to choose a set $S$ of measure $c$ that is
  located as far from $\mu$ as possible. Thus, the worst case set
  $S = (-\infty,-z]\cup[z,\infty)$ consists of both tails of
  $\normal(0,\matr{I})$, each having mass $c/2$. We know that for the CDF of
  the normal distribution, the following holds:
  \[ \Phi(z)\ge 1-e^{-\frac{z^2}{2}} \]
  and we need to find the point $z$ such that $\Phi(z)=1-\frac{c}{2}$. Thus, we
  have:
  \[ \frac{c}{2}\leq e^{-\frac{z^2}{2}}\Leftrightarrow z\leq \sqrt{2\log \left(\frac{2}{c}\right)} \]
  It remains to upper bound
  $\Exp_{y \sim \normal(0, 1, S)}\left[ (y-\vec{w}^T\vec{x})^2 \right]$.

    For the aforementioned worst case set $S$, this quantity is equal to the
  second non-central moment around $\vec{w}\vec{x}$ of the truncated Gaussian
  distribution in the interval $[z,\infty)$. Thus,
  \[ \Exp_{y \sim \normal(0, 1, S)}\left[ (y-\vec{w}^T\vec{x})^2 \right] = 1 + \frac{z \phi(z)}{1-\Phi(z)} \]
  \noindent The term $M(z) = \frac{\phi(z)}{1-\Phi(z)}$ is the inverse Mills
  ratio which is bounded by $z \le M(z) \le z+1/z$ for $z > 0$, see
  \cite{gordon41}.

    Thus,
  $z^2 + 1 \le \Exp_{y \sim \normal(0, 1, S)}\left[ (y-\vec{w}^T\vec{x})^2 \right] \le z^2 + 2 \le 2\log \left(\frac{2}{c}\right) + 2$.

\subsection{Proof of Lemma \ref{lem:rejectionSamplingLemma}}
\label{sec:rejectionSamplingLemmaProof}

  Using the second part of Lemma \ref{lem:firstLemma} we have that for every
$\vec{w} \in \Domain_{r^*, \bar{B}}$ it holds that
\begin{align*}
  \log \left(\frac{1}{\alpha(\vec{w}, \vec{x}^{(i)})}\right) & \le 2 \log \left(\frac{1}{\alpha(\vec{w}, \vec{x}^{(j)})}\right) + \abs{\vec{w}^T \left( \vec{x}^{(i)} - \vec{x}^{(j)} \right)}^2 + 2
\end{align*}
\noindent if we also we the fact that $\norm{\vec{x}^{(i)}}_2 \le 1$ and
$\norm{\vec{w}}_2 \le \bar{B}$, from \eqref{eq:normalizedAssumptions} and
Definition \ref{def:regression:setaki} we get that
\begin{align} \label{eq:proof:lem:rejectionSamplingLemma:1}
  \log \left(\frac{1}{\alpha(\vec{w}, \vec{x}^{(j)})}\right) & \ge \frac{1}{2} \log \left(\frac{1}{\alpha(\vec{w}, \vec{x}^{(i)})}\right) - \bar{B} - 1.
\end{align}

  Using the first part of Lemma \ref{lem:firstLemma} we have that
\begin{align*}
  \log \left(\frac{1}{\alpha(\vec{w}, \vec{x}^{(i)})}\right) & \le 2 \log \left(\frac{1}{\alpha(\vec{w}^*, \vec{x}^{(i)})}\right) + \abs{\left( \vec{w} - \vec{w}^* \right)^T \vec{x}^{(i)}}^2 + 2 \\
  & \le 2 \log \left(\frac{1}{\alpha(\vec{w}^*, \vec{x}^{(i)})}\right) + \abs{y^{(i)} - \vec{w}^T \vec{x}^{(i)}}^2 + \abs{y^{(i)} - \vec{w}^{*T} \vec{x}^{(i)}}^2 + 2
\end{align*}
\noindent Let now $\vec{v} \in \reals^k$ be any unit vector, then we can
multiply the above inequality by $\left(\vec{v}^T \vec{x}^{(i)}\right)^2$ and
sum over all $i \in [n]$ and we get that
\begin{align*}
  \sum_{i = 1}^n \left(\vec{v}^T \vec{x}^{(i)}\right)^2 \log \left(\frac{1}{\alpha(\vec{w}, \vec{x}^{(i)})}\right) & \le 2 \sum_{i = 1}^n \left(\vec{v}^T \vec{x}^{(i)}\right)^2 \log \left(\frac{1}{\alpha(\vec{w}^*, \vec{x}^{(i)})}\right) + \\
  & ~~~~~~~~ + \sum_{i = 1}^n \left(\vec{v}^T \vec{x}^{(i)}\right)^2 \abs{y^{(i)} - \vec{w}^T \vec{x}^{(i)}}^2 + \\
  & ~~~~~~~~ + \sum_{i = 1}^n \left(\vec{v}^T \vec{x}^{(i)}\right)^2 \abs{y^{(i)} - \vec{w}^{*T} \vec{x}^{(i)}}^2 + 2 \\
  \intertext{now we can use the fact that $\vec{w} \in \Domain_{r^*, \bar{B}}$,
  $\vec{w}^* \in \Domain_{r^*, \bar{B}}$, and Assumption
  \ref{asp:survivalProbability} to get}
  \sum_{i = 1}^n \frac{\left(\vec{v}^T \vec{x}^{(i)}\right)^2}{\sum_{j = 1}^n \left(\vec{v}^T \vec{x}^{(j)}\right)^2} \log \left(\frac{1}{\alpha(\vec{w}, \vec{x}^{(i)})}\right) & \le 2 \log\left(\frac{1}{a}\right) + 2 r^* + 2. \\
  \intertext{Now fix some $\ell \in [n]$, using
  \eqref{eq:proof:lem:rejectionSamplingLemma:1} for the indices $\ell$ and $i$
  for all $i$ in the above inequality we get that}
  \frac{1}{2} \log \left(\frac{1}{\alpha(\vec{w}, \vec{x}^{(\ell)})}\right) & \le 2 \log\left(\frac{1}{a}\right) + 2 r^* + 3 + \bar{B}
\end{align*}
\noindent from which the lemma follows.

\subsection{Proof of Lemma \ref{lem:thirdLemma}} \label{sec:thirdLemmaProof}

    By our assumption we have that bounded $\ell_2$ norm condition of
  $\Domain_{r^*, \bar{B}}$ is satisfied. So our goal next is to prove that the
  first condition is satisfied too. Hence, we have to prove that
  \[ \frac{1}{n} \sum_{i = 1}^n \left(\vec{\eps}^{(i)}\right)^2
     \vec{x}^{(i)} \vec{x}^{(i) T} \preceq r \frac{1}{n} \sum_{i = 1}^n \vec{x}^{(i)} \vec{x}^{(i) T}, \]
  \noindent where $\vec{\eps}^{(i)} = y^{(i)} - \vec{w}^{*T} \vec{x}^{(i)}$
  and $r = \log\left(\nicefrac{1}{a}\right)$. Observe that
  $\vec{\eps}^{(i)}$ is a truncated standard normal random variable with the
  following truncation set
  \[ S^{(i)} = \left\{\vec{z} \mid \left( \vec{z} + \vec{w}^{*T} \vec{x}^{(i)}
     \right) \in S \right\}. \]
  \noindent Therefore we have to prove that for any unit vector $\vec{v}$ it
  holds that
  \[ \frac{1}{n} \sum_{i = 1}^n \left(\vec{\eps}^{(i)}\right)^2
     \left( \vec{v}^T \vec{x}^{(i)} \right)^2 \le r \cdot \frac{1}{n} \sum_{i = 1}^n \left( \vec{v}^T \vec{x}^{(i)} \right)^2. \]
  \noindent For start we fix a unit vector $\vec{v}$ and we want to bound the
  following probability
  \[ \Prob\left( \frac{1}{n} \sum_{i = 1}^n \left(\vec{\eps}^{(i)}\right)^2
     \left( \vec{v}^T \vec{x}^{(i)} \right)^2 \ge t \cdot \frac{1}{n} \sum_{i = 1}^n \left( \vec{v}^T \vec{x}^{(i)} \right)^2 \right). \]
  \noindent This means that we are interested in the independent random
  variables $\left( \vec{\eps}^{(i)} \right)^2$. Let $q^{(i)}$ the real value
  that is guaranteed to exist from Lemma
  \ref{lem:truncatedGaussianToExponential} and corresponds to
  $\vec{\eps}^{(i)}$ and $\vec{u}^{(i)}$ the corresponding random variable
  guaranteed to exist from Lemma \ref{lem:truncatedGaussianToExponential}. We
  also define
  $\vec{\delta}^{(i)} = \left( \vec{\eps}^{(i)} \right)^2 - q^{(i)}$, finally
  set $a^{(i)} = \alpha(\vec{w}^*, \vec{x}^{(i)})$. We have that
  \begin{align*}
    \frac{1}{n} \sum_{i = 1}^n \left(\vec{\eps}^{(i)}\right)^2
     \left( \vec{v}^T \vec{x}^{(i)} \right)^2 & \le
    2 \frac{1}{n} \sum_{i = 1}^n \vec{\delta}^{(i)}
     \left( \vec{v}^T \vec{x}^{(i)} \right)^2 +
    2 \frac{1}{n} \sum_{i = 1}^n q^{(i)}
     \left( \vec{v}^T \vec{x}^{(i)} \right)^2 \\
     & \le 2 \frac{1}{n} \sum_{i = 1}^n \vec{\delta}^{(i)}
      \left( \vec{v}^T \vec{x}^{(i)} \right)^2 +
     4 \frac{1}{n} \sum_{i = 1}^n \log\left( \frac{2}{a^{(i)}} \right)
      \left( \vec{v}^T \vec{x}^{(i)} \right)^2
    \intertext{but from Assumption \ref{asp:survivalProbability} this implies}
    & \le 2 \frac{1}{n} \sum_{i = 1}^n \vec{\delta}^{(i)}
     \left( \vec{v}^T \vec{x}^{(i)} \right)^2 +
    4 \log\left( \frac{2}{a} \right).
  \end{align*}
  \noindent Now the random variables $\vec{\delta}^{(i)}$ are stochastically
  dominated by $\vec{u}^{(i)}$ and hence
  \[ \Prob\left( \frac{1}{n} \sum_{i = 1}^n \vec{\delta}^{(i)}
     \left( \vec{v}^T \vec{x}^{(i)} \right)^2 \ge t \cdot \frac{1}{n} \sum_{i = 1}^n \left( \vec{v}^T \vec{x}^{(i)} \right)^2 \right) \le
     \Prob\left( \frac{1}{n} \sum_{i = 1}^n \vec{u}^{(i)}
        \left( \vec{v}^T \vec{x}^{(i)} \right)^2 \ge t \cdot \frac{1}{n} \sum_{i = 1}^n \left( \vec{v}^T \vec{x}^{(i)} \right)^2 \right)  \]
  \noindent But from the fact that $\Exp\left[\vec{u}^{(i)}\right] = 1/2$ we
  have that
  \begin{align*}
    \frac{1}{n} \sum_{i = 1}^n \vec{u}^{(i)}
     \left( \vec{v}^T \vec{x}^{(i)} \right)^2 & \le
     \frac{1}{n} \sum_{i = 1}^n \left( \vec{u}^{(i)} -
        \Exp\left[ \vec{u}^{(i)} \right]\right)
        \left( \vec{v}^T \vec{x}^{(i)} \right)^2 +
     \frac{1}{n} \sum_{i = 1}^n \left( \vec{v}^T \vec{x}^{(i)} \right)^2
  \end{align*}
  \noindent Therefore we have that
  \begin{align} \label{eq:secondToLastProbability}
    \Prob\left( \frac{\sum_{i = 1}^n \left(\vec{\eps}^{(i)}\right)^2
       \left( \vec{v}^T \vec{x}^{(i)} \right)^2}{\sum_{i = 1}^n \left( \vec{v}^T \vec{x}^{(i)} \right)^2} \ge t +
       4 \log\left(\nicefrac{2}{a}\right) + 1 \right) \le
       \Prob\left( \frac{\sum_{i = 1}^n \left( \vec{u}^{(i)} -
          \Exp\left[ \vec{u}^{(i)} \right]\right)
          \left( \vec{v}^T \vec{x}^{(i)} \right)^2}{\sum_{i = 1}^n \left( \vec{v}^T \vec{x}^{(i)} \right)^2} \ge t \right).
  \end{align}

  \noindent To bound the last probability we will use the matrix analog of
  Bernstein's inequality as we expressed in Theorem \ref{thm:matrixBernstein}.
  More precisely we will bound the following probability
  \[ \Prob\left( \lambda_{\max} \left( \sum_{i = 1}^n \left( \vec{u}^{(i)} -
     \Exp\left[ \vec{u}^{(i)} \right]\right)
     \vec{x}^{(i)} \vec{x}^{(i) T} \right) \ge t \cdot \lambda_{\max} \left( \sum_{i = 1}^n \vec{x}^{(i)} \vec{x}^{(i) T} \right) \right).
  \]
  \noindent We set
  $\matr{Z}_i = \frac{1}{n \cdot \lambda_{\max}(\matr{X})} \left( \vec{u}^{(i)} - \Exp\left[ \vec{u}^{(i)} \right]\right) \vec{x}^{(i)} \vec{x}^{(i) T}$,
  where $\matr{X} = \frac{1}{n} \sum_{i = 1}^n \vec{x}^{(i)} \vec{x}^{(i) T}$.
  Observe that
  \[ \left( \vec{x}^{(i)} \vec{x}^{(i) T} \right)^p =
     \norm{\vec{x}^{(i)}}_2^{2(p - 2)}
       \left( \vec{x}^{(i)} \vec{x}^{(i) T} \right)^2. \]
  \noindent Also since $\vec{u}^{(i)}$ is an exponential random variable with
  parameter $1/2$, it is well known that
  \[ \Exp\left[ \left( \vec{u}^{(i)} -
        \Exp\left[ \vec{u}^{(i)} \right]\right)^p \right] \le \frac{p!}{2} 2^p.
  \]
  \noindent These two relations imply that
  \[ \Exp\left[ \matr{Z}_i^p \right] \preceq \frac{p!}{2} \left( 2 \cdot \frac{1}{\lambda_{\max}(\matr{X})} \cdot
        \frac{\norm{\vec{x}^{(i)}}_2^2}{n} \right)^{p - 2} \cdot
          \left( \frac{1}{\lambda_{\max}(\matr{X})} \cdot \frac{2}{n} \vec{x}^{(i)} \vec{x}^{(i) T} \right)^2 \]
  \noindent Hence we can apply Theorem \ref{thm:matrixBernstein} with
  $R = 2 \cdot \frac{1}{\lambda_{\max}(\matr{X})} \cdot \frac{\norm{\vec{x}^{(i)}}_2^2}{n}$ and
  $\matr{A}_i = \frac{1}{\lambda_{\max}(\matr{X})} \cdot \frac{2}{n} \vec{x}^{(i)} \vec{x}^{(i) T}$.
  Now observe that by Assumption \ref{asp:logkAssumption} we get that
  \[ \frac{\norm{\vec{x}^{(i)}}_2^2}{n} \le \frac{\lambda_{\max}(\matr{X})}{\log k}. \]
  We now compute the variance parameter
  \[ \sigma^2 = \norm{\sum_{i = 1}^n \matr{A}_i} = 4 \cdot \frac{1}{\lambda_{\max}(\matr{X})} \cdot
        \frac{\norm{\vec{x}^{(i)}}_2^2}{n}
          \norm{\frac{1}{n} \sum_{i = 1^n} \vec{x}^{(i)}} \vec{x}^{(i) T} \le
            \frac{4}{\log k}. \]
  Using the same reasoning we get $R \le 2/\log k$, therefore applying Theorem
  \ref{thm:matrixBernstein} we get that
  \begin{align*}
    \Prob\left( \lambda_{\max} \left( \sum_{i = 1}^n \matr{Z}_i \ge t \right)
    \right)
        & \le k \exp \left( - \frac{t^2}{2 \sigma^2 + R t} \right) \\
        & \le k \exp \left( - \frac{t^2 \log k}{8 + 2 t}\right)
  \end{align*}
  \noindent From the last inequality we get that for $t \ge 5$ there is at most
  probability $1/3$ such that
  \[ \Prob\left( \lambda_{\max} \left( \sum_{i = 1}^n \left( \vec{u}^{(i)} -
     \Exp\left[ \vec{u}^{(i)} \right]\right)
     \vec{x}^{(i)} \vec{x}^{(i) T} \right) \ge t \cdot \lambda_{\max}(\matr{X}) \right) \le \frac{1}{3}
  \]
  \noindent and hence the lemma follows.

\subsection{Proof of Theorem \ref{thm:conditionsSatisfied}}
\label{sec:thm:conditionsSatisfiedProof}

\noindent We now use our Lemmas \ref{lem:firstLemma}, \ref{lem:secondLemma},
\ref{lem:thirdLemma} to prove some useful inequalities that we need to prove
our strong convexity and bounded step variance.

\begin{lemma} \label{lem:usefulInequalities}
    Let $(\vec{x}^{(1)}, y^{(i)}), \dots, (\vec{x}^{(n)}, y^{(n)})$
  be $n$ samples from the linear regression model
  \eqref{eq:truncatedLinearRegressionDefinition} with parameters $\vec{w}^*$. If
  Assumptions \ref{asp:survivalProbability} and \ref{asp:logkAssumption} hold,
  and if $\norm{\vec{x}^{(i)}}_2 \le 1$ for all $i \in [n]$ then
  \begin{align} \label{eq:thm:useful1}
    \frac{1}{n} \sum_{i = 1}^n \norm{((\vec{w} - \vec{w}^*)^T
    \vec{x}^{(i)}) \vec{x}^{(i)}}_2^2 \le 2 \cdot \bar{B}^2,
  \end{align}
  \begin{align} \label{eq:thm:useful2}
    \frac{1}{n} \sum_{i = 1}^n \Exp \left[ \norm{(y^{(i)} - \vec{w}^{*T}
    \vec{x}^{(i)}) \vec{x}^{(i)}}_2^2 \right] \le
    2 \log\left( \frac{1}{a} \right) + 4,
  \end{align}
  \begin{align} \label{eq:thm:useful3}
    \frac{1}{n} \sum_{i = 1}^n \Exp \left[ \norm{(z^{(i)} - \vec{w}^T
    \vec{x}^{(i)}) \vec{x}^{(i)}}_2^2 \right] \le r^* + 2 \cdot \bar{B}^2,
  \end{align}
  \noindent where $y^{(i)} \sim \normal(\vec{w}^{*T} \vec{x}^{(i)}, 1, S)$,
  $z^{(i)} \sim \normal(\vec{w}^T \vec{x}^{(i)}, 1, S)$ and
  $\vec{w} \in \Domain_{r^*, \bar{B}}$ with
  $r^* = 4 \log \left( \nicefrac{2}{a} \right) + 7$.
\end{lemma}

\begin{proof}    We prove one by one the above inequalities.
  \smallskip

  \noindent \textbf{Proof of \eqref{eq:thm:useful1}.}
  \noindent Since $\vec{w}, \vec{w}^* \in \Domain_{r^*, \bar{B}}$ we have that
  $\norm{\vec{w} - \vec{w}^*}_2 \le 2 \bar{B}$. If we combine this with the
  assumption that $\norm{\vec{x}^{(i)}}_2 \le 1$ and the fact that
  $\norm{\vec{z}}_2 \le \norm{\vec{z}}_{\infty} \cdot \sqrt{k}$ then we get that
  \[ \frac{1}{n} \sum_{i = 1}^n \norm{((\vec{w} - \vec{w}^*)^T
  \vec{x}^{(i)}) \vec{x}^{(i)}}_2^2 \le
  2 \cdot B^2 \cdot k. \]
  \smallskip

  \noindent \textbf{Proof of \eqref{eq:thm:useful2}.} Using Lemma
  \ref{lem:secondLemma} we get that
  \[ \frac{1}{n} \sum_{i = 1}^n \Exp \left[ \norm{(y^{(i)} - \vec{w}^{*T}
    \vec{x}^{(i)}) \vec{x}^{(i)}}_2^2 \right] \le \frac{1}{n}
    \sum_{i = 1}^n \left(2 \log\left( \frac{1}{\alpha(\vec{w}^{*T},
    \vec{x}^{(i)})} \right) + 4 \right) \norm{\vec{x}^{(i)}}_2^2 \]
  \noindent but now we can use Assumption \ref{asp:survivalProbability} and
  the assumption that $\norm{\vec{x}^{(i)}}_2 \le 1$ and we get that
  \[ \frac{1}{n} \sum_{i = 1}^n \Exp \left[ \norm{(y^{(i)} - \vec{w}^{*T}
    \vec{x}^{(i)}) \vec{x}^{(i)}}_2^2 \right] \le
    2 \log\left( \frac{1}{a} \right) + 4. \]
  \smallskip

  \noindent \textbf{Proof of \eqref{eq:thm:useful3}.} Using Lemma
  \ref{lem:secondLemma} we get that
  \[ \frac{1}{n} \sum_{i = 1}^n \Exp \left[ \norm{(z^{(i)} - \vec{w}^T
    \vec{x}^{(i)}) \vec{x}^{(i)}}_2^2 \right] \le \frac{1}{n}
    \sum_{i = 1}^n \left(2 \log\left( \frac{1}{\alpha(\vec{w},
    \vec{x}^{(i)})} \right) + 4 \right) \norm{\vec{x}^{(i)}}_2^2. \]
  \noindent Using Lemma \ref{lem:firstLemma} on the last inequality we get
  that
  \begin{align*}
    \frac{1}{n} \sum_{i = 1}^n \Exp \left[ \norm{(z^{(i)} - \vec{w}^T
    \vec{x}^{(i)}) \vec{x}^{(i)}}_2^2 \right] & \le \frac{1}{n}
    \sum_{i = 1}^n \left(2 \log\left( \frac{1}{\alpha(\vec{w}^*,
    \vec{x}^{(i)})} \right) + 4 \right) \norm{\vec{x}^{(i)}}_2^2 + \\
    & ~~~ + \frac{1}{n} \sum_{i = 1}^n \norm{((\vec{w} - \vec{w}^*)^T
    \vec{x}^{(i)}) \vec{x}^{(i)}}_2^2 + 2 \\
    \intertext{Now using \eqref{eq:thm:useful1}, Assumption
    \ref{asp:survivalProbability}, and the facts that
    $\norm{\vec{x}^{(i)}}_2 \le 1$ and
    $\norm{\vec{w} - \vec{w}^*}_2 \le 2 \bar{B}$ we get that}
    frac{1}{n} \sum_{i = 1}^n \Exp \left[ \norm{(z^{(i)} - \vec{w}^T
    \vec{x}^{(i)}) \vec{x}^{(i)}}_2^2 \right] & \le 2 \log\left( \frac{1}{a} \right) + 4 + 2 \cdot \bar{B}^2 \\
    & \le r^* + 2 \cdot \bar{B}^2.
  \end{align*}
\end{proof}

\noindent Given Lemma \ref{lem:usefulInequalities} we are now ready to prove
Theorem \ref{thm:conditionsSatisfied}.

\begin{proof}[Proof of Theorem \ref{thm:conditionsSatisfied}]
     We first prove \eqref{eq:thm:condition1} and then
  \eqref{eq:thm:condition2}.
  \smallskip

  \noindent \textbf{Proof of \eqref{eq:thm:condition1}.} Using the fact that
  $\norm{\vec{x}^{(i)}}_2 \le 1$ we get that
  \begin{align*}     \Exp\left[ \norm{y^{(i)} \vec{x}^{(i)} - z^{(j)} \vec{x}^{(j)}}_2^2 \right] & \le \Exp\left[ \norm{y^{(i)} \vec{x}^{(i)}}_2^2 + \norm{z^{(j)} \vec{x}^{(j)}}_2^2 \right] \le \Exp\left[\left( y^{(i)} \right)^2\right] + \frac{1}{n} \sum_{j \in [n]} \Exp\left[\left( z^{(j)} \right)^2\right] \\
    & \le \Var\left[y^{(i)}\right] + \left(\Exp[y^{(i)}]\right)^2 + \frac{1}{n} \sum_{j \in [n]} \left( \Var\left[z^{(j)}\right] + \left(\Exp[z^{(j)}]\right)^2 \right)
  \end{align*}
  \noindent Now we can use the fact that
  $y^{(i)} \sim \mathcal{N}(\vec{w}^{*T} \vec{x}^{(i)}, 1, S)$ and that
  $z^{(j)} \sim \mathcal{N}(\vec{w}^{T} \vec{x}^{(j)}, 1, S)$ together with
  Lemma 6 from \cite{DaskalakisGTZ18} and the fact
  $\norm{\vec{x}^{(i)}}_2 \le 1$ we get that
  \begin{align*}     \Exp\left[ \norm{y^{(i)} \vec{x}^{(i)} - z^{(j)} \vec{x}^{(j)}}_2^2 \right] & \le 2 \norm{\vec{w}^*}^2_2 + 3 \log\left(\frac{1}{\alpha(\vec{w}^*, \vec{x}^{(i)})}\right)
    + 2 \norm{\vec{w}}^2_2 + \frac{3}{n} \sum_{j \in [n]} \log\left(\frac{1}{\alpha(\vec{w}, \vec{x}^{(j)})}\right)
  \end{align*}

  \noindent Now using Lemma \ref{lem:rejectionSamplingLemma} and the fact that
  $\norm{\vec{w}^*}_2 \le \bar{B}$ and $\norm{\vec{w}}_2 \le \bar{B}$ we get
  that
  \begin{align*}
    \Exp\left[ \norm{y^{(i)} \vec{x}^{(i)} - z^{(j)} \vec{x}^{(j)}}_2^2 \right] & \le O(\poly(1/a) \cdot \bar{B}^2)
  \end{align*}
  \noindent from which \eqref{eq:thm:condition1} follows.
  \medskip

  \noindent \textbf{Proof of \eqref{eq:thm:condition2}.} Let $\vec{v}$ be an
  arbitrary unit vector in $\reals^k$. Using Lemma \ref{lem:fourthLemma}
  from Appendix \ref{sec:app:aux} and then the first part of Lemma
  \ref{lem:firstLemma}, we have that
  \begin{align*}
    \frac{1}{n} \sum_{i = 1}^n \Exp \left[ \left( z^{(i)} -
    \Exp\left[ z^{(i)} \right] \right)^2 \right] &
    \left( \vec{v}^T \vec{x}^{(i)} \right)^2 \ge \frac{1}{n} \sum_{i = 1}^n \alpha^2(\vec{w}, \vec{x}^{(i)}) \left( \vec{v}^T \vec{x}^{(i)} \right)^2 \\
    & \ge \frac{e^{-2}}{n} \sum_{i = 1}^n \exp\left( - \norm{(\vec{w} - \vec{w}^*)^T \vec{x}^{(i)}}_2^2 - 2 \log\left(\frac{1}{\alpha(\vec{w}^*, \vec{x}^{(i)})} \right) \right) \cdot \\
    & ~~~~~~~~~~~~~~~~ \cdot \left( \vec{v}^T \vec{x}^{(i)} \right)^2
    \intertext{now we can divide both sides with
    $\frac{1}{n} \sum_{i = 1}^n \left( \vec{v}^T \vec{x}^{(i)} \right)^2$ and
    then apply Jensen's inequality and we get}
    \frac{1}{\frac{1}{n} \sum_{i = 1}^n \left( \vec{v}^T \vec{x}^{(i)} \right)^2} \frac{1}{n} & \sum_{i = 1}^n \Exp \left[ \left( z^{(i)} -
    \Exp\left[ z^{(i)} \right] \right)^2 \right]
    \left( \vec{v}^T \vec{x}^{(i)} \right)^2 \\
    & \ge \exp \left(- \frac{\sum_{i = 1}^n \left( \norm{(\vec{w} - \vec{w}^*)^T
    \vec{x}^{(i)}}_2^2 + 2 \log \left(\frac{1}{\alpha(\vec{w}^*, \vec{x}^{(i)})} \right)
    \right) \left(\vec{v}^T \vec{x}^{(i)}\right)^2}{\sum_{i = 1}^n \left( \vec{v}^T \vec{x}^{(i)} \right)^2} \right)
    \intertext{now applying Assumption \ref{asp:survivalProbability} we get}
    & \ge a^2 \exp \left(- \frac{\sum_{i = 1}^n
     \norm{(\vec{w} - \vec{w}^*)^T \vec{x}^{(i)}}_2^2
     \left(\vec{v}^T \vec{x}^{(i)}\right)^2}{\sum_{j = 1}^n \left(\vec{v}^T \vec{x}^{(i)}\right)^2} \right) \\
    & \ge a^2 \exp \left(- \frac{\sum_{i = 1}^n
     \norm{y^{(i)} - \vec{w}^{*T} \vec{x}^{(i)}}_2^2
     \left(\vec{v}^T \vec{x}^{(i)}\right)^2}{\sum_{j = 1}^n \left(\vec{v}^T \vec{x}^{(i)}\right)^2} \right) \cdot \\
    & ~~~~~~~~~~~ \cdot \exp \left(- \frac{\sum_{i = 1}^n
     \norm{y^{(i)} - \vec{w}^T \vec{x}^{(i)}}_2^2
     \left(\vec{v}^T \vec{x}^{(i)}\right)^2}{\sum_{j = 1}^n \left(\vec{v}^T \vec{x}^{(i)}\right)^2} \right)
  \end{align*}
  \noindent now we can use the fact the both $\vec{w}$ and $\vec{w}^*$
  belong to the projection set $\Domain_{r^*, B}$ as we showed in Lemma
  \ref{lem:thirdLemma} and we get

  \begin{align*}
    \frac{1}{\frac{1}{n} \sum_{i = 1}^n \left( \vec{v}^T \vec{x}^{(i)} \right)^2} \frac{1}{n} \sum_{i = 1}^n \Exp \left[ \left( z^{(i)} -
    \Exp\left[ z^{(i)} \right] \right)^2 \right]
    \left( \vec{v}^T \vec{x}^{(i)} \right)^2 \ge
    a^2 \exp\left(- 2 r^* \right)
  \end{align*}
  \noindent and using Assumption \ref{asp:logkAssumption} the relation
  \eqref{eq:thm:condition2} follows.
\end{proof}

\section{Efficient Sampling of Truncated Gaussian with Union of Intervals Survival Set} \label{sec:unionIntervalsSampling}

  In this section, in Lemma \ref{lem:samplingUnionOfIntervalsLemma}, we see
that when $S = \cup_{i = 1}^r [a_i, b_i]$, with $a_i, b_i \in \mathbb{R}$,
then we can efficiently sample from the truncated normal distribution
$\normal(\mu, \sigma, S)$ in time $\poly(n, k)$ even under the weak bound on
the mass of the set $\normal(\mu, \sigma; S) \ge a^{\poly(n, k)}$. The only
difference is that instead of exact sampling we have approximate sampling, but
the approximation error is exponentially small in total variation distance and
hence it cannot affect any algorithm that runs in poly-time.

\begin{definition}[\textsc{Evaluation Oracle}] \label{def:evaluationOracle}
    Let $f : \mathbb{R} \to \mathbb{R}$ be an arbitrary real function.
  We define the \textit{evaluation oracle} $\mathcal{E}_f$ of $f$ as an
  oracle that given a number $x \in \mathbb{R}$ and a target accuracy
  $\eta$ computes an $\eta$-approximate value of $f(x)$, that is
  $\left|\mathcal{E}_f(x) - f(x)\right| \le \eta$.
\end{definition}

\begin{lemma} \label{lem:inversionLemma}
    Let $f : \mathbb{R} \to \mathbb{R}_+$ be a real increasing and
  differentiable function and $\mathcal{E}_f(x)$ an evaluation oracle of
  $f$. Let $\ell \le f'(x) \le L$ for some $\ell, L \in \mathbb{R}_+$. Then we can construct an algorithm that implements the evaluation
  oracle of $f^{-1}$, i.e. $\mathcal{E}_{f^{-1}}$. This implementation on
  input $y \in \mathbb{R}_+$ and input accuracy $\eta$ runs in time
  $T$ and uses at most $T$ calls to the evaluation oracle $\mathcal{E}_{f}$
  with inputs $x$ with representation length $T$ and input accuracy
  $\eta' = \eta/\ell$, with
  $T = \poly\log(\max\{|f(0)/y|, |y/f(0)|\}, L, 1/\ell, 1/\eta)$.
\end{lemma}

\begin{proof}[Proof of Lemma \ref{lem:inversionLemma}]
    Given a value $y \in \mathbb{R}_+$ our goal is to find an
  $x \in \mathbb{R}$ such that $f(x) = y$. Using doubling we can find
  two numbers $a, b$ such that $f(a) \le y - \eta'$ and
  $f(b) \ge y + \eta'$ for some $\eta'$ to be determined later. Because of
  the lower bound $\ell$ on the derivative of $f$ we have that this step
  will take $\log((1/\ell) \cdot \max\{|f(0)/y|, |y/f(0)|\})$ steps.
  Then we can use binary search in the interval $[a, b]$ where in each step
  we make a call to the oracle $\mathcal{E}_f$ with accuracy $\eta'$ and we
  can find a point $\hat{x}$ such that
  $\left| f(x) - f(\hat{x}) \right| \le \eta'$. Because of the upper bound
  on the derivative of $f$ we have that $f$ is $L$-Lipschitz and hence this
  binary search will need $\log(L/\eta')$ time and oracle calls. Now
  using the mean value theorem we get that for some $\xi \in [a, b]$ it
  holds that
  $\left| f(x) - f(\hat{x}) \right| = |f'(\xi)| \left|x - \hat{x} \right|$
  which implies that $\left|x - \hat{x}\right| \le \eta'/\ell$, so if we
  set $\eta' = \ell \cdot \eta$, the lemma follows.
\end{proof}

  Using the Lemma \ref{lem:inversionLemma} and the Proposition 3 of
\cite{Chevillard12} it is easy to prove the following lemma.

\begin{lemma} \label{lem:samplingOneIntervalLemma}
    Let $[a, b]$ be a closed interval and $\mu \in \mathbb{R}$ such that
  $\gamma_{[a, b]}(\mu) = \alpha$. Then there exists an algorithm that runs
  in time $\poly\log(1/\alpha, \zeta)$ and returns a sample of a
  distribution $\mathcal{D}$, such that
  $d_{\mathrm{TV}}(\mathcal{D}, N(\mu, 1; [a, b])) \le \zeta$.
\end{lemma}

\begin{proof}[Proof Sketch]
    The sampling algorithm follows the steps:
  (1) from the cumulative distribution function $F$ of the distribution
  $N(\mu, 1; [a, b])$ define a map from $[a, b]$ to $[0, 1]$, (2)
  sample uniformly a number $y$ in $[0, 1]$ (3) using an evaluation oracle
  for the error function, as per Proposition 3 in \cite{Chevillard12},
  and using Lemma \ref{lem:inversionLemma} compute with exponential
  accuracy the value $F^{-1}(y)$ and return this as the desired sample.
\end{proof}

  Finally using again Proposition 3 in \cite{Chevillard12} and Lemma
\ref{lem:samplingOneIntervalLemma} we can get the following lemma.

\begin{lemma} \label{lem:samplingUnionOfIntervalsLemma}
    Let $[a_1, b_1]$, $[a_2, b_2]$, $\dots$, $[a_r, b_r]$ be closed
  intervals and $\mu \in \mathbb{R}$ such that
  $\gamma_{\cup_{i = 1}^r [a_i, b_i]}(\mu) = \alpha$. Then there exists
  an algorithm that runs in time $\poly(\log(1/\alpha, \zeta), r)$ and
  returns a sample of a distribution $\mathcal{D}$, such that
  $d_{\mathrm{TV}}(\mathcal{D}, N(\mu, 1; \cup_{i = 1}^r [a_i, b_i])) \le \zeta$.
\end{lemma}

\begin{proof}[Proof Sketch]
    Using Proposition 3 in \cite{Chevillard12} we can compute
  $\hat{\alpha}_i$ which estimated with exponential accuracy the mass
  $\alpha_i = \gamma_{[a_i, b_i]}(\mu)$ of every interval $[a_i, b_i]$.
  If $\hat{\alpha}_i/\alpha \le \zeta/3r$ then do not consider interval
  $i$ in the next step. From the remaining intervals we can choose one
  proportionally to $\hat{\alpha}_i$. Because of the high accuracy in the
  computation of $\hat{\alpha}_i$ this is $\zeta/3$ close in total
  variation distance to choosing an interval proportionally to $\alpha_i$.
  Finally after choosing an interval $i$ we can run the algorithm of Lemma
  \ref{lem:samplingOneIntervalLemma} with accuracy $\zeta/3$ and hence the
  overall total variation distance from
  $N(\mu, 1; \cup_{i = 1}^r [a_i, b_i])$ is at most $\zeta$.
\end{proof}

\end{document}